\documentclass[preprint,aos]{imsart}
\RequirePackage{natbib}

\setattribute{journal}{name}{}

\RequirePackage[OT1]{fontenc}
\RequirePackage{amsthm,amsmath,natbib,graphicx,enumerate}
\RequirePackage[colorlinks,citecolor=blue,urlcolor=blue]{hyperref}
\RequirePackage{hypernat}
\usepackage{amssymb,tabularx,multicol,multirow,booktabs}
\usepackage{mhequ}
\usepackage{relsize}
\usepackage[toc,page]{appendix}
\startlocaldefs
\numberwithin{equation}{section}
\theoremstyle{plain}
\newtheorem{thm}{Theorem}[section]
\endlocaldefs
\newtheorem{theorem}[thm]{Theorem}

\newtheorem{lemma}[thm]{Lemma}
\newtheorem{prop}[thm]{Proposition}

\newtheorem{defn}[thm]{Definition}
\newtheorem{remark}[thm]{Remark}

\theoremstyle{remark}

\def \be{\begin{equs}}
\def \ee{\end{equs}}

\def \I{\mathcal{I}}

\def \E{\mathbb{E}_0}

\def \tn {\sqrt{2\log{n}}}

\def \ihat {\hat{I}}
\def \ihatk {\ihat(k_0,k_*)}
\def \Ik {I(k_0,k_*)}
\def \k0 {k_0}
\def \kstar {k_*}
\def \k1 {k_1}
\def \k2 {k_2}
\def \k3 {k_3}
\def \jhat {\hat{j}}
\def \Pf {\mathbb{P}_f}
\def \Ef {\mathbb{E}_f}
\def \psilk {\psi_{l,k}}
\def \psilprimek {\psi_{l^{'},k}}
\def \alphalk {\alpha_{l,k}}
\def \alphalprimek {\alpha_{l^{'},k}}
\def \sumlk {\sum_{l=1}^{k}}

\def \sumllprime {\sum_{l \neq l^{'}}}
\def \psimax {\|\psi_0\|_{\infty}}
\def \lprime {l^{'}}
\def \fmax {\|f\|_{\infty}}
\def \const {C_{\psi_0}}
\def \constf {C(\psi_0,\fmax)}
\def \sumin {\sum_{i=1}^n}
\def \sumjn {\sum_{j=1}^n}
\def \sumij {\sum_{i \neq j}}
\def \sumijs {\sum_{i \neq j\neq s}}
\def \Ei {\mathbb{E}_{f_{X_i}}}
\def \Ej {\mathbb{E}_{f_{X_j}}}

\def \Ejs {\mathbb{E}_{f_{X_j},f_{X_s}}}

\def \n2 {\frac{1}{n(n-1)}}
\def \diag {\chi_{n,m} \times \chi_{n,m}}
\def \diagthree {\chi_{n,m} \times \chi_{n,m}\times 
	\chi_{n,m}}
\def \unionm {\bigcup_{m=1}^{M_n}}
\def \pnm {p_{n,m}}
\def \chinm {\chi_{n,m}}
\def \xij {(X_i,X_j)}
\def \xijs {(X_i,X_j,X_s)}
\def \In {\mathbf{I}_n}
\def \vmone {V_{n,m}^{(1)}}
\def \vmtwo {V_{n,m}^{(2)}}
\def \vmone {V_{n,m}^{(1)}}
\def \vmthree {V_{n,m}^{(3)}}
\def \vmcond {\Ef(V_{n,m}|\In)}
\def \summ {\sum_{m=1}^{M_n}}

\def \Nm {N_{n,m}}
\def \fmin {f_{\mathrm{min}}}

\let\tilde\widetilde

\def \nthree {\frac{1}{n(n-1)(n-2)}}
\def \ktilde {\tilde{K}}
\numberwithin{equation}{section}
\begin{document}
%\small
%\doublespacing
\begin{frontmatter}
\title{Lepski's Method and Adaptive Estimation of Nonlinear Integral Functionals of Density}
%\runtitle{Covariance matrix estimation}
%\thankstext{T1}{Footnote to the title with the `thankstext' command.}
\begin{aug}
	\author{\fnms{Rajarshi} \snm{Mukherjee}\thanksref{t1,m1}\ead[label=e1]{rmukherj@stanford.edu}}
	\author{\fnms{Eric} \snm{Tchetgen Tchetgen}\thanksref{t2,m2}\ead[label=e2]{etchetge@hsph.harvard.edu}}
	\and
	\author{\fnms{James} \snm{Robins}\thanksref{t3,m1}\ead[label=e3]{robins@hsph.harvard.edu}}
	\thankstext{t1}{Stein Fellow, Department of Statistics, Stanford University}
	\thankstext{t2}{Professor, Department of Biostatistics, Harvard University}
	\thankstext{t3}{Professor, Department of Epidemiology, Harvard University}

\end{aug}
\begin{abstract}We study the adaptive minimax estimation of non-linear integral functionals of a density and extend the results obtained for linear and quadratic functionals to general functionals. The typical rate optimal non-adaptive minimax estimators of "smooth" non-linear functionals are higher order U-statistics. Since Lepski's method requires tight control of tails of such estimators, we bypass such calculations by a modification of Lepski's method which is applicable in such situations. As a necessary ingredient, we also provide a method to control higher order moments of minimax estimator of cubic integral functionals. Following a standard constrained risk inequality method, we also show the optimality of our adaptation rates.
\end{abstract}


\end{frontmatter}

\section*{Introduction}
Estimation of statistical functionals in nonparametric problems has received considerable attention over the last few decades. Of specific interest have been both linear and non-linear integral functionals of an underlying density. For example, a large body of work has focused on estimating the entropy of an underlying distribution. \cite{beirlant1997nonparametric} provides an overview of results and related techniques. More recent works include estimation of R\'{e}nyi and Tsallis entropies \citep{leonenko2010statistical,pal2010estimation}. For more references and examples one can refer to \cite{kandasamy2014influence}. We consider a framework for estimating such integral functionals of a density. In particular, suppose $X_1,\ldots,X_n$ are i.i.d on $[0,1]$ with density $f(x)$ with respect to Lebesgue measure. %We will eventually provide discussions about extending our methods to $[0,1]^d$. However, since the crux of the idea can be laid out in the one-dimensional case as well, we keep such efforts for later sections.
We take $f \in H(\beta,C)$ where $H(\beta,C)$ is a H\"{o}lder ball of smoothness $\beta$ and radius $C>0$. We are interested in estimation of $\phi(f)$ where $\phi: \mathcal{F}: \rightarrow \mathbb{R}$ is a non-linear functional of density and $\mathcal{F}$ refers to the class of all densities on $[0,1]$. It is well known \citep{birge1995estimation} that if $\phi(f)=\int T(f(x))d\mu(x)$, and $T$ is sufficiently smooth, then the minimax rate of estimation over $f \in H(\beta,C)$ in squared error norm is  $n^{-\frac{8\beta}{1+4\beta}}$ when $\beta < \frac{1}{4}$ and $n^{-1}$ for $\beta \geq \frac{1}{4}$. 

There exits extensive literature addressing minimax estimation of linear and quadratic functionals in density, white noise or nonparametric regression models. Although definitely not exhaustive, a comprehensive snapshot of  this immense body of work can be found in \cite{ hall1987estimation, bickel1988estimating, donoho1990minimax1, donoho1990minimax2, fan1991estimation, kerkyacharian1996estimating, laurent1996efficient, cai2003note, cai2004minimax, cai2005nonquadratic} and other references therein. Most estimators proposed in the literature above, which attain the minimax rate of convergence over certain smoothness classes of the underlying function, depend explicitly on the knowledge of the smoothness index of the class. In particular, a standard technique in estimation of these functionals is expanding the infinite dimensional function of interest in an suitable orthonormal basis of $L_2[0,1]$ and estimate an approximate functional created by truncating the basis expansion at certain point. The point of truncation decides the approximation error of the truncated functional as an surrogate for the actual functional and depends on the smoothness of the function of interest and approximation properties of the orthonormal basis used. This point of truncation is then delicately balanced with the bias and variance of the resulting estimator and therefore directly depends on the smoothness of the function. Thus, it becomes of interest to understand the question of adaptive estimation i.e. the construction and analysis of estimators without prior knowledge of the smoothness.

The question of adaptation of linear and quadratic functionals has been studied in detail as well \citep{low1992renormalization, efromovich1994adaptive, efromovich1996optimal,efromovich2000adaptive, klemela2001sharp, laurent2000adaptive, cai2005adaptive1,cai2005adaptive2,cai2006adaptation,cai2006optimal,cai2008adaptive, gine2008simple} and references therein. However, adaptive estimation of general non-linear functionals has not been addressed in complete generality. In this paper we address the problem of adaptive estimation of general smooth non-linear functionals of a density. 

\cite{robins2008higher} have developed a concrete theory for addressing minimax estimation of a class of non-linear functionals in certain non-parametric and semi-parametric problems under low regularity smoothness conditions. Following similar logic of construction \cite{tchetgen2008minimax} constructs a minimax estimator of $\int f^3 d\mu$ for $\beta <\frac{1}{4}$. The specific minimax estimator is a third order U-statistics and the construction of the kernel depends explicitly on the knowledge of the underlying smoothness. Based on the technique of \cite{birge1995estimation} one can show that the minimax estimator of a general non-linear functional $\phi(f)=\int T(f)d\mu$ for smooth $T$ can be constructed by using ideas from linear, quadratic and cubic functionals of the density and appealing to a standard Taylor expansion argument.  While producing adaptive estimators of non-linear functionals as well, a similar strategy can be followed and it becomes crucial to understand the adaptive estimation of $\int f^3 d\mu$.

In our search for answers regarding $\int f^3 d\mu$, we start with looking at the general idea driving estimation of quadratic functionals of the density i.e. $\phi(f)=\int f^2 d\mu$. \cite{gine2008simple} provide adaptive estimators of quadratic functional $\phi(f)$ where the estimators are based on certain types of second order U-statistics with specific kernels. Our results provide a richer class of estimators based on a compactly supported wavelet basis and is in line with the theory established by \cite{gine2008simple, efromovich1996optimal}. 

Unlike a quadratic functional, the minimax estimator for $\phi(f)=\int f^r  d\mu$ is in general a $r^{\mathrm{th}}$ order U-statistic. In order to apply regular Lepski's Method we will need to obtain suitable exponential deviation inequalities for higher order U-statistics. Although, such moment inequalities do exist \citep{adamczak2006moment}, the bounds include complicated quantities which needs to be controlled in a problem specific manner. To bypass such complications, we employ a modification of Lepski's method where we test for the smoothness using our previously obtained second order U-statistics and use the selected smoothness to estimate the required functional. 

In this paper we focus on adaptation over the non$-\sqrt{n}$ regime i.e. $\beta<\frac{1}{4}$; although our proofs carries over easily for $-\sqrt{n}$ range i.e. $\beta \geq \frac{1}{4}$ as well where adaptation is possible without paying a price and it is possible to achieve asymptotic efficiency for $\beta>\frac{1}{4}$. Moreover, the case of higher dimensions of $X$ i.e. $d>1$ can also be achieved by arguments similar to those in this paper. A brief discussion of these possible extensions is given in Section 8.

The main contributions of this paper are as followed. This work extends previous results obtained for linear and quadratic functionals to new adaptation theory for non-linear integral functionals of the density. In order to do so, we develop a suitable variant of Lepski's method which bypasses establishing exponential tail bounds for estimators of general non-linear integral functionals. In applying this modified Lepski's method, the main challenge lies in obtaining suitable bounds on higher order moments of suitable estimators of general non-linear integral functionals. This requires control of moments of higher order U-statistics based on orthogonal projection kernels. We crucially use the structure of the projection kernels based on compactly supported wavelet basis. Following ideas from \cite{rltv2015}, we implement a binning argument to keep track of membership of sampled observations in partitions of the sample space created by a particular resolution level of the wavelet expansion, and this turns out to be crucial in controlling the higher order moments of our estimator at the right level.

We would like to mention that our proofs for analyzing adaptive estimator of general nonlinear integral functional of the density uses projection kernels based on Haar Basis expansion. We provide proof for general compactly supported wavelet based procedure when analyzing the quadratic functional. However, for analysis regarding general nonlinear integral functional of the density, the requirement of Haar Basis arises at one specific instance in the proof. We expect that similar results continue to hold for more general compactly supported wavelets using arguments developed herein. However we do not further consider other wavelet bases here. 

The paper is organized as follows. In Section 1, we discuss Lepski's method the main idea of the paper. In Section 2,  we introduce basic notations, terminology and some theory about compactly supported wavelets we will need in the sequel. In Section 3, we study the estimation of the quadratic functional and provide a general class of adaptive estimators when the kernels of the U-statistics are based on a class of compactly supported wavelet bases. Section 4 is devoted to the development of a modified Lepski's method suitable for analyzing higher order moments of the density. We then use this results to develop an adaptive estimator of $\int f^3 d\mu$ in section 5. Section 6 is used to understand how construction of an adaptive estimator of a general non-linear functional $\int T(f)d\mu$ for smooth $T$ can be derived from our analyses in Section 5. A lower bound on the required price of adaptation is provided in Section 7. Section 8 contains some discussions and future work. Finally all proofs are collected in the Appendix.

\section{Lepski's Method and Heuristics of the Main Idea}\hspace*{\fill} \\
The purpose of this section is to heuristically explain the main idea behind modifying Lepski's Method suitably in our context. We do this at level of abstraction and do not provide formal results. Sections 4 and 5 are devoted to making the heuristics of this section more precise. We begin with a discussion of Lepski's Method as a recipe for producing adaptive estimators from a sequence of candidate estimators. Our discussion is inspired by the wonderful article of \cite{birge2001alternative}. %We elaborate on the underlying principles once again since this helps us explain the idea behind our modification of Lepski's Method suitable for this context. 
As astutely observed by \cite{birge2001alternative}, Lepski's Method can be succinctly described in a relatively abstract set up as follows.

Using notations in essence similar to \cite{birge2001alternative}, consider a family of experiment spaces $\{(\chi,\mathcal{B}(\chi),P),P \in S_{\theta}\}_{\theta \in \Theta}$ for measurable space $\chi$, sigma field $\mathcal{B}(\chi)$ and probability measure $P$ in one of the parameter sets $\{S_{\theta}\}_{\theta \in \Theta}$. 
With an i.i.d sample of size $n$ from such an experiment, we consider uniform rates of convergence for estimation of certain objects of interest over these parameter sets. In particular, suppose $s: P \rightarrow \mathcal{Y}$ is an object of interest for the experiment of interest with $\mathcal{Y}$ being a pseudo-metric space equipped with pseudo-metric $d(\cdot,\cdot)$. For a given estimator $\hat{s}$ based on sample of size $n$ of a required object of interest $s$, define its rate of convergence over $S_{\theta}$ as $r(\theta,\hat{s})=\sup_{s \in S_{\theta}}\E\left(d^q(s,\hat{s})\right)$, where $q \geq 1$. An estimator $\hat{s}$ is often called (minimax) rate optimal over $S_{\theta}$ if  $r(\theta,\hat{s})\asymp r(\theta):=\inf_{\tilde{s}}r(\theta,\tilde{s})$.

Now, starting from a family of rate optimal estimators $\{\hat{s}_{\theta}\}_{\theta \in {\Theta}}$, Lepski's Method provides a strategy for building a new estimator $\hat{\theta}$ which has ``good" performance simultaneously over all sets $S_{\theta},\theta \in \Theta$. Working under the regime of estimating a whole function in the context of nonparametric regression, \cite{lepskii1991problem, lepskii1992asymptotically, lepskii1993asymptotically} develop such a strategy. Later these methods have been further used by \cite{efromovich1994adaptive, efromovich1996optimal,klemela2001sharp,gine2008simple} and others for studying adaptation theory of linear and quadratic functionals in nonparametric problems. The method can be summarized roughly in brief as follows.

Suppose that $\Theta \subset \mathbb{R}$ is a bounded subset such that $S_{\theta}$ is non-decreasing with respect to $\theta$, the risks and minimax rates $r(\theta,\hat{s}_{\theta}), r(\theta)$ are continuous with respect to $\theta$, and for each $\theta \in \Theta$, $\exists$ a rate optimal estimator $\hat{s}_{\theta}$ where for large enough $n$, $d^q(s,\hat{s}_{\theta})$ is suitably concentrated around its expectation. One then chooses, for each $n$, a suitable fine discretization $\theta_1<\ldots<\theta_{K(n)}$ of $\Theta$ and finally for some large enough constant $C$ define the candidate estimator to be $\hat{s}_{\theta_{\jhat}}$ where
$$\jhat:=\inf \left\{j \leq K(n): d^q(\hat{s}_{\theta_j},\hat{s}_{\theta_l})\leq C r(\theta_l,\hat{s}_{\theta_l}), \ \forall l \in (j,K(n))\right\}.$$

Let us try to intuitively elaborate on the method described above. In its heart, the definition of $\jhat$ is devoted to choosing the ``best" $\theta$ from a a point of view of the risk of the estimator of interest. Often, the ``best" $\theta$ is the one corresponding to the unknown data generating mechanism and if $\jhat$ selects a value ``close enough" to this $\theta$ with high probability, then owing to the continuity property of the risk, one obtains desired adaptive performance of the final estimator $\hat{s}_{\theta_{\jhat}}$.  The required high probability selection of the desired $\theta$ is driven by the  concentration of $d^q(s,\hat{s}_{\theta})$ around its expectation. 

Consider now a situation where suitable concentration of $d^q(s,\hat{s}_{\theta})$ around its expectation is not easily achieved. This is often the case when the sequence of estimators $\{\hat{s}_{\theta}\}_{\theta \in \Theta}$ is not sufficiently ``nice" for application of standard concentration inequalities. In such cases there can be two possible ways out. The first entails obtaining a sufficiently sharp concentration inequality for the sequence of estimators $\{\hat{s}_{\theta}\}_{\theta \in \Theta}$ at hand. This is often an important and difficult question in its own right, even outside the context of the problem of adaptive estimation at hand. In certain cases however, an alternative strategy might be available. This paper  pursues the second approach in the context of estimating a non-linear integral functional of a density.

At a high level of abstraction, our method can be described as follows. As mentioned earlier, the main challenge in applying Lepski's Method is obtaining a suitably sharp  concentration of $d^q(s,\hat{s}_{\theta})$ around its expectation. However, suppose that there exists another sequence of estimators of a different object of interest $\tilde{s}: P \rightarrow \tilde{\mathcal{Y}}$ with $\tilde{\mathcal{Y}}$ being a pseudo-metric space equipped with pseudo-metric $\tilde{d}(\cdot,\cdot)$ and corresponding sequence of estimators $\hat{\tilde{s}}_{\theta}$ with concentration of $\tilde{d}^q(\tilde{s},\hat{\tilde{s}}_{\theta})$ around its expectation and a further property that $\tilde{d}^q(\hat{\tilde{s}}_{\theta},\hat{\tilde{s}}_{\theta'})$ and $\tilde{r}(\theta,\hat{\tilde{s}})$ equals ${d}^q(\hat{{s}}_{\theta},\hat{{s}}_{\theta'})$ and $r(\theta,\hat{s})$ up to multiplicative constants uniformly in $(\theta,\theta')\in \Theta \times \Theta$. Our idea is to define 

$$\tilde{\jhat}:=\inf \left\{j \leq K(n): \tilde{d}^q(\hat{\tilde{{s}}}_{\theta_j},\hat{\tilde{{s}}}_{\theta_l})\leq C \tilde{r}(\theta_l,\hat{\tilde{{s}}}_{\theta_l}), \ \forall l \in (j,K(n))\right\},$$
and using $\hat{s}_{\theta_{\jhat}}$ as our candidate estimator. Under some conditions on the actual sequence of estimators $\hat{s}_{\theta}$ one then obtains desired adaptation properties of the final estimator $\hat{s}_{\theta_{\jhat}}$. These conditions on $\hat{s}_{\theta}$ needs to be less demanding than sharp concentration of $d^q(s,\hat{s}_{\theta})$ around its expectation, and this turns out to be the case in our context.

To fix ideas, for our case the experiment space will be $(\chi,\mathcal{B}(\chi),P: P \in S_{\theta})_{\theta}=([0,1],\mathcal{B}([0,1]),P=fd\mu: f \in H(\beta,C))_{\beta \in (0,\frac{1}{4})}$; i.e. parameter spaces of interest are $H(\beta,C)$ where $\beta$ is identified with $\theta$ in the discussion above and the object of interest is $s=\int T(f) d\mu$ based on i.i.d. observations $X_1,\ldots,X_n$ from density $f\in H(\beta,C)$. As discussed earlier, the case of quadratic functionals., i.e. $T(x)=x^2$ has been studied in detail in the literature. In particular, one has suitable concentration of $\tilde{d}^2(\int f^2 d\mu, \tilde{\hat{s}}_{\beta})$ around its expectation, where $\tilde{d}(z_1,z_2)=|z_1-z_2|$ for $z_1,z_2 \in \mathbb{R}$ and $\tilde{\hat{s}}_{\beta}$ is rate optimal estimator of the quadratic functional. Such an estimator is often a second order U-statistic, for which the desired concentration type results can be obtained in reference to \cite{gine2000exponential,houdre2003exponential}. However, for estimating $s=\int f^3 d\mu$, which is crucial for understanding the general smooth non-linear functional problem, the estimator sequence $\{s_{\beta}\}$ is a third order U-statistic. In general,  for estimating $s=\int f^r d\mu$, the estimator sequence $\{s_{\beta}\}$ is a $r^{\mathrm{th}}$ order U-statistic. In order to apply the standard Lepski's Method we will need to obtain suitable exponential deviation inequalities for higher order U-statistics. Although, such moment inequalities exist \citep{adamczak2006moment}, the bounds include complicated quantities that must be controlled in a problem specific manner. However, owing to \cite{tchetgen2008minimax}, we observe that the biases of the estimator sequence $\{s_{\beta}\}$ are within multiplicative constants of that of the estimator for quadratic functional i.e. $\{\tilde{s}_{\beta}\}$ and their variances remain of same order as well. Since we are interested in estimation in squared error norm, this motivates us to use $\tilde{\hat{j}}$ constructed using the estimator for quadratic functional and analyzing $s_{\beta_{\tilde{\hat{j}}}}$ as our final estimator.

In Sections 3 and 4 below, we elaborate on the ideas laid down above. In particular, we provide a concrete example of the application of standard Lepski's Method while estimating quadratic functional of the density. Subsequently, we employ the modification of Lepski's method as discussed above to study adaptive estimation of more general non-linear integral functionals of the density. 
\section{Notation}
 It is well known that without imposing restrictions on function classes, consistent inference in non-parametric problems is not feasible. Our results are stated in terms of certain regularity conditions on the class of densities. We will place the following kind of bounds on their roughness or complexity. 

\begin{defn}
	A function $h(\cdot )$ with domain $\left[ 0,1\right] ^{d}$ is said to
	belong to a H\"{o}lder ball $H(\beta ,C),$ with H\"{o}lder exponent $\beta
	>0 $ and radius $C>0,$ if and only if $h\left( \cdot \right) $ is uniformly
	bounded by $C$, all partial derivatives of $h(\cdot )$ up to order $%
	\left\lfloor \beta \right\rfloor $ exist and are bounded, and all partial
	derivatives $\nabla ^{\left\lfloor \beta \right\rfloor }$ of order $%
	\left\lfloor \beta \right\rfloor $ satisfy 
	\begin{equation*}
		\sup_{x,x+\delta x\in \left[ 0,1\right] ^{d}}\left\vert \nabla
		^{\left\lfloor \beta \right\rfloor }h(x+\delta x)-\nabla ^{\left\lfloor
			\beta \right\rfloor }h(x)\right\vert \leq C||\delta x||^{\beta -\left\lfloor
			\beta \right\rfloor }.
	\end{equation*}
\end{defn}
%All our results are in terms of H\"{o}lder balls. However, they are easily extendable to Sobolev balls . 
With some abuse of notation we will denote by $H(\beta)$ the class of functions in $H(\beta,C)$ such that $\fmax \leq M^*$ and $\fmin \geq M_*$ (where $f_{\min}=\inf_{x\in [0,1]}f(x)$) for known constants $M^*>M_*>0$. Note that, indeed knowledge of $C$ gives us an bound on $M^*$. Our results will be based on known  $C,M^*,M_*$ only to the extent that these quantities will determine multiplicative constants of convergence rates obtained throughout. Adaptation w.r.t. these parameters may be of interest, however beyond the scope of this paper.%Although Besov bodies require more careful calculations, the ideas governing the adaptation theory is very similar.

A crucial ingredient for constructing our estimators is the use of orthogonal projection kernels onto increasing finite dimensional subspaces of $L_2[0,1]$. The construction of such kernels is based on a suitable choice of an orthonormal basis of $L_2[0,1]$. With this in mind we provide a brief discussion on orthogonal projection kernels and compactly supported wavelet bases.

The kernels (possibly dependent on $n$) are assumed to be measurable maps $K_n: [0,1]\times [0,1] \rightarrow \mathbb{R}$ that are symmetric in their arguments and satisfy $\int \int K_n^2(x_1,x_2)dx_1 dx_2<\infty$ for all $n$. The corresponding kernel operator (which we denote by the same symbol with an abuse of notation) 
$$K_n h(x)=\int h(v)K_n(x,v)dv$$
are continuous linear operators $K_n: L_2[0,1]\rightarrow L_2[0,1]$. Throughout we will work with kernels whose operator norms $\|K_n\|=\sup\{\|K_n f\|_2: \|f\|_2\leq 1\}$ are uniformly bounded, i.e., $\sup_n \|K_n\|<\infty$. %This is certainly the case if $K_n f \rightarrow f$ in $L_2[0,1]$ for all $f \in L_2[0,1]$ (by the Banach-Steinhaus Theorem).
 The operator norm $\|K_n\|$ is however typically much smaller than the $L_2[0,1]\times L_2[0,1]$ norm of the kernel. In our case this will typically be of the order of $k$ given by the dimension of the projected space.  

Kernels $K_n$ most commonly used in statistical applications are usually projection kernels. A projection kernel operator satisfies $K_n h=h$ for all $h$ in its range space; that is for any function $f$ in the range of the kernel,  $f(x)=\int f(v)K_n(x,v)dv$ for $a.e. \ x$. For a given orthonormal basis $e_1,e_2,\ldots$ of $L_2[0,1]$, the orthogonal projection onto $lin(e_1,\ldots,e_k)$ is the map $K_k: f \rightarrow \sum_{j=1}^k \langle f, e_j \rangle e_j$. It is given by the kernel operator $K_k: L_2[0,1] \rightarrow L_2[0,1]$ with kernel
$$K_k(x_1,x_2)=\sum_{j=1}^k e_j(x_1)e_j(x_2).$$
It is easy to show by orthonormality properties that $\|K_k\|=1$ and $\int \int K_k^2 =k$. %We will always work with orthogonal projection kernels as above.

We now provide examples of orthogonal projection kernels that will be used extensively throughout. 

\subsection{Haar Basis}
For ``father" and ``mother" functions $\phi(x)=\I(0,1](x)$ and $\psi(x)=\I(0,1/2](x)-\I(1/2,1](x)$, the Haar basis is the set of functions $\{\phi,\psi_{i,j}:i=0,1,\ldots,j=0,1,\ldots, 2^i-1\}$, for $\psi_{i,j}(x)=2^{i/2}\psi(2^i x-j)$, $i=0,1,\ldots,j=0,1,\ldots, 2^i-1$. The Haar basis is a complete orthonormal basis of $L_2[0,1]$. The linear span of the first $k=2^I$ basis elements $\{\phi,\psi_{i,j}: i=0,\ldots,I-1, j=0,\ldots,2^i-1\}$ is equal to the linear span of the scaled and shifted father functions $\{\psi_{I,j}: j=0,\ldots, 2^I-1\}$ given by the kernel
\be
\ & K_k(x_1,x_2)=\sum_{j=0}^{2^I-1}\phi_{I,j}(x_1)\phi_{I,j}(x_2)\\&=k\sum_{j=1}^{k}I\left(x_1\in \left(\frac{j-1}{k},\frac{j}{k}\right]\right)I\left(x_2\in \left(\frac{j-1}{k},\frac{j}{k}\right]\right).
\ee

\subsection{Wavelets}
Consider expansions of functions $h \in L_2(\mathbb{R})$ on an orthonormal basis of compactly supported bounded wavelets of the form 

\be
h(x)&=\sum_{j \in \mathbb{Z}^d}\sum_{v \in \{0,1\}^d}\langle h, \psi_{0,j}^v\rangle \psi_{0,j}^v(x)+ \sum_{i=0}^{\infty}\sum_{j \in \mathbb{Z}^d}\sum_{v \in \{0,1\}^d-\{0\}}\langle h, \psi_{i,j}^v\rangle \psi_{i,j}^v(x) ,\\ \label{eqn:wavelet_expansion}
\ee
where the base functions $\psi_{i,j}^v$ are orthogonal for different indices $(i,j,v)$ and are scaled and translated versions of the $2^d$ base functions $\psi_{0,0}^v$, i.e., $\psi_{i,j}^v(x)=2^{id/2}\psi_{0,0}^v(2^i x-j)$. Such a wavelet basis can be obtained as tensor products $\psi_{0,0}^v=\phi^{v_1}\times \ldots \times \phi^{v_d}$ of a given father wavelet $\phi^0$ and mother wavelet $\phi^1$ in one dimension. 

We are interested in functions $f$ with support $[0,1]$. In view of the compact support of the wavelets, for each resolution level $i$ and index $v$, only $2^i$ base elements $\psi_{i,j}^v$ are non-zero on $[0,1]$; let us denote the corresponding set of indices $j$ by $j_i$. Truncating the expansion at resolution level $i=I$ then gives an orthogonal projection on a subspace of dimension $k$ of order $2^{I}$. Let 
$$K_k(x_1,x_2)=\sum_{j \in J_0}\sum_{v \in \{0,1\}}\psi_{0,j}^v(x_1)\psi_{0,j}^v(x_2)+\sum_{i=0}^{I}\sum_{j \in J_i}\psi_{i,j}^1(x_1)\psi_{i,j}^1(x_2).$$ It is worth noting that we can re-express the wavelet expansion \eqref{eqn:wavelet_expansion} to start from a level $I$ as 

\be
h(x)&=\sum_{j \in \mathbb{Z}}\sum_{v \in \{0,1\}}\langle h, \psi_{I,j}^v\rangle \psi_{I,j}^v(x)+ \sum_{i=I+1}^{\infty}\sum_{j \in \mathbb{Z}}\langle h, \psi_{i,j}^1\rangle \psi_{i,j}^1(x) .
\ee
The projection kernel $K_k$ sets the coefficients in the second sum equal to zero and hence can also be expressed as 
 
 $$K_k(x_1,x_2)=\sum_{j \in J_I}\sum_{v\in \{0,1\}}\psi_{I,j}^v(x_1)\psi_{I,j}^v(x_2).$$

Owing to the discussion above, with some abuse of notation, we will work with the definition of projection kernel as 
$$K_k(x_1,x_2)=\sum_{j =1}^{k}\psi_{k,j}(x_1)\psi_{k,j}(x_2),$$
which will correspond to the orthogonal projection operator onto the first $O(\log_2 k)$ wavelet basis. Owing to compactness of the support of wavelet bases only a fixed $O(k)$ many $\psi_{k,j}$'s will intersect  $[0,1]$. We have then conveniently renamed them as $1,\ldots,k$. This simplifies notation without loss of generality. In particular, this is exact for Haar kernel since as previously noted, we then have that 
$$K_k(x,y)=\sum_{l=1}^k\psi_{k,j}(x)\psi_{k,j}(y)$$
where $\psi_{k,j}(x)=\sqrt{k}\I(x \in (\frac{j-1}{k},\frac{j}{k})],\ j=1,\ldots,k$ is the $\log_2{k}$ level Haar basis of $L_2[0,1]$. Also, for larger expressions involved in some proofs, we often write $\psi_{k,l}=\psi_l^k$ for convenience.

With these conventions, let us now discuss some approximation properties of these projection kernels. Letting $\overline{\psi}:=lin\{\psi_{k,j}: j=1,\ldots,k\}$, define
$$f_k(x):=\Pi_{\mathrm{Lebesgue}}\left(f|\overline{\psi}\right):=\int f(y)K_k(x,y)dy.$$ 
For convenience of notation, define
$$\alpha_{k,j}=\int{\psi_{k,j}f}.$$

It is well known that \citep{hardle1998wavelets}, choosing $\psi_{k,l}$'s to be the $\log_2 k$ level compactly supported wavelet basis with suitable vanishing moment conditions on the mother wavelet, one has that $$\sup\limits_{h \in H(\beta,C)}\|h-h_k\|_{2} \lesssim k^{-\beta}.$$

The results in this paper are mostly asymptotic in nature and thus requires some standard asymptotic  notations. If $a_n$ and $b_n$ are two sequences of real numbers then $a_n \gg b_n$ (and $a_n \ll b_n$) implies that ${a_n}/{b_n} \rightarrow \infty$ (and ${a_n}/{b_n} \rightarrow 0$) as $n \rightarrow \infty$, respectively. Similarly $a_n \gtrsim b_n$ (and $a_n \lesssim b_n$) implies that $\liminf{{a_n}/{b_n}} = C$ for some $C \in (0,\infty]$ (and $\limsup{{a_n}/{b_n}} =C$ for some $C \in [0,\infty)$). Alternatively, $a_n=o(b_n)$ will also imply $a_n \ll b_n$ and $a_n=O(b_n)$ will imply that $\limsup{{a_n}/{b_n}} =C$ for some $C \in [0,\infty)$).
Finally we comment briefly on the various constants appearing throughout the text and proofs. Given that our primary results concern convergence rates of various estimators, we will not emphasize the role of constants throughout and rely on fairly generic notation throughout for such constants.
% We choose not be very particular about tracking constants and use similar names for many constants appearing in several places. 
Some conventions we follow whenever required is a follows. Throughout the paper we denote by $\constf$ a non-negative constant that depends on a fixed $\psi_0$ (a function that can be taken as a majorant of both father and mother wavelets in absolute value) and $\fmax$. Often, $\const$ will denote a number which depends only on $\psi_0$. Hence if $\fmax$ is known to us, such constants are deterministic and eventually can all be replaced by an universal constant by taking suitable care of all possible constants appearing in this paper. Finally, we will also use $C(\psi_0,\fmax, \fmin)$ to denote numbers which depends on $\psi_0,\fmax, \fmin$ only.

 \section{Lepski's Method and Minimax Adaptive Estimation of $\phi(f)=\int f^2 d\mu$}\hspace*{\fill} \\
%Lepski's Method in our context can be broken down into the following steps.
%\begin{enumerate}[(i)]
%	\item First assume $\beta \in \{\beta_0,\beta_1\} \subset (0,\frac{1}{4})$ and try to understand the subtlety of adaptation over the two point set.
%	\item Assume $\beta \in \{\beta_0,\beta_1,\ldots,\beta_M\} \subset (0,\frac{1}{4})$ for some fixed $M >1$ and try to understand adaptation over finitely many points.
%	\item For adaptation over the whole of $(0,\frac{1}{4})$ we then explain how suitable discretization of $(0,\frac{1}{4})$ takes us back to step 2.
%	\end{enumerate}

By way of introduction, this section we will provide a concrete example where Lepski's Method applies in its standard form. For the sake of simplicity, we will restrict ourselves to the adaptation over two parameter spaces indexed by two smoothness classes $\beta_1<\beta_0<\frac{1}{4}$. The more general case can be addressed from our results in Section 3. However, before going into further details, we need additional  notation and a basic result which drives the main idea behind Lepski's Method. For our candidate sequence of estimators, define
$$U_n^{(k)}=%\V\left(K_k(X_{i_1},X_{i_2})\right):
\frac{1}{n(n-1)}\sum\limits_{i_1\neq i_2}K_k(X_{i_1},X_{i_2}),$$
$$k(\beta)=\lceil n^{\frac{2}{1+4\beta}}\rceil,$$
$$U_{n,j}:=U_n^{(k_j)},\ k_j=k(\beta_j):=\lceil n^{\frac{2}{1+4\beta_j}}\rceil,$$
$$U_n^{(k_j^*)},\ k_j^*=k^*(\beta_j)=\lceil \left(\frac{n^2}{\log{n}}\right)^{\frac{1}{1+4\beta_j}}\rceil.$$

It is easy to show that, knowing the exact smoothness $\beta_j$ one can readily use $U_{n,j}$ as a minimax rate optimal estimator of $\phi(f)$. %However, if one is trying to test between $\beta \in \{\beta_0,\beta_1,\ldots,\beta_M\}$, then one needs to appropriately design a data dependent testing procedure to test for the underlying smoothness. Lepski's method is a cleverly designed yet intuitive testing procedure which performs this. 
Asymptotic normality of the estimators, as stated by Lemma \ref{lemma_asymp_norm}, drives Lepski's Method as described in Section 2. 
%\section*{Two point adaptation: $\beta \in \{\beta_0,\beta_1\} \subset (0,\frac{1}{4})$}\hspace*{\fill} \\
%Let us illustrate the strategy when $\beta \in \{\beta_0,\beta_1\}$ where $\beta_1<\beta_0<\frac{1}{4}$.
Let,
$$U_{n,0}:=U_n^{(k_0)},\ k_0=\lceil n^{\frac{2}{1+4\beta_0}}\rceil,$$
$$U_{n,1}:=U_n^{(k_1)},\ k_1=\lceil n^{\frac{2}{1+4\beta_1}}\rceil,$$
$$U_{n,*}:=U_n^{(k_*)},\ k_*=\lceil \left(\frac{n^2}{\log{n}}\right)^{\frac{1}{1+4\beta_1}}\rceil.$$
%As noted earlier, provided one knows that the smoothness is $\beta_j$ one can readily use $U_{n,j}$ as their estimator of $\phi(f)$. However, if one is trying to test between $\beta \in \{\beta_0,\beta_1\}$, then using $U_{n,0}$ when the truth is $\beta_1$ incurs a polynomial penalty. However, in order to balance testing errors appropriately, we will use $U_{n,*}$ appropriately.

%\subsection*{{Intuitive Explanation of the Method}}
 The method we now describe, is along the lines of \cite{efromovich1996optimal} suitably adapted to deal with our situation. Our goal is to choose a data dependent $\hat{k} \in \{k_0,k_1\}$ so that if we base our estimation on $U_n^{(\hat{k})}$, the mean squared error of estimation will be suitably controlled. In particular, 
\begin{eqnarray*}
	&&\sup_{f}E\left[ U_{n,\widehat{j}}-\psi \left( f\right) \right] ^{2} \\
	&\leq &\sum\limits_{k=0}^{1}\sup_{f}E\left\{ \I\left\{ \widehat{j}%
	=k\right\} \left[ U_{n,\widehat{j}}-\psi \left( f\right) \right] ^{2}\right\}
	\\
	&\equiv & \sup_f R_{0}(\beta_f)+\sup_f R_{1}(\beta_f)
\end{eqnarray*}
We divide our study into two cases having two sub-cases each.

%\subsubsection*{$\beta=\beta_0$}
First, consider the case $\beta=\beta_0$. Then, $R_0(\beta_0)$ achieves the minimax rate since the truth is $\beta_0$ and we are choosing $\beta_0$. Next we need to bound $R_1(\beta_0)$ by $n^{-\frac{8\beta_0}{1+4\beta_0}}$ up to a potentially logarithmic factor. It is worth noting, that along the lines of \cite{efromovich1996optimal}, since $\beta_0$ corresponds to higher smoothness regime, we do not need to pay the logarithmic price. The term $R_1(\beta_0)$ corresponds to the scenario where the truth is $\beta_0$ but the procedure wrongly chooses $\beta_1$ as the underlying smoothness. If this happens too often, then the mean squared error will be sub-optimal. Therefore, our testing procedure should guarantee that type of error, i.e., choosing a lower smoothness when the truth is a higher smoothness, does not happen too often. Since, the probability of making such an error must reduce the mean squared error $n^{-\frac{8\beta_1}{1+4\beta_1}}$ down to $n^{-\frac{8\beta_0}{1+4\beta_0}}$ for any $\beta_1$, it must be of $O(\frac{1}{n})$. So the problem boils down to designing a selection procedure so that probability of selecting the lower smoothness is $O(\frac{1}{n})$. If one can find a sequence of statistics $T_n$ such that $T_n$ converges weakly to N(0,1) "fast enough"  under $\beta_0$ and $|T_n|$ diverges polynomially in $n$ under $\beta_1$, then a simple test for selecting the lower smoothness with required error rate is given by $\I(T_n>\tn)$. To see this, note that, the probability of selecting the lower smoothness under $\beta_0$ is approximately $Pr(|N(0,1)|>\tn)=O(\frac{1}{n})$. Provided the approximation in the CLT is also of the order of $\frac{1}{n}$ we have achieved our goal for this first case.

Considering the complementary case of $\beta=\beta_1$, we will need the error probability of selecting the wrong smoothness $Pr(|T_n|\leq \tn)$ to converge to $0$ ``polynomially fast". As we establish below, this can be achieved with the right choice of test statistic. Finally, by Lemma \ref{lemma_asymp_norm}, a candidate $T_n$ which satisfies the above mentioned properties is given by  $$\frac{U_{n,*}-U_{n,0}}{\sqrt{Var_f(U_{n,*})}}.$$

 We now provide the exact technical details that justify the argument laid down so far. %The main challenge, as we shall see, is showing that the convergence to normality in Lemma \ref{lemma_asymp_norm} is fast enough.  

%\subsection*{\textbf{Technical Details for Two Point Adaptation of Quadratic Functionals}}\hspace*{\fill} \\
We start with a few further notations. Define
$$I(k_1,k_2):=\phi(f_{k_2})-\phi(f_{k_1})\ k_1\leq k_2,$$
corresponding to the difference in truncation bias at $k_1$ versus $k_2$.
With this notation, it is easy to see that $I(k_0,k_*)=O(k_0^{-2\beta_f})$ when the truth is $\beta_f$. A suitable estimator of $I(k_0,k_*)$ is indeed $$\hat{I}(k_0,k_*)=U_{n,*}-U_{n,0}.$$ 
%Finally define our estimator as 
%\be\label{estimator}
%U_{n,\hat{j}}, \ \hat{j}:=\I\left(\ihatk>\sqrt{\frac{\kstar 2\log{n}}{n^2}}\right). 
%\ee
%We show,
%$$\sup_{f \in H(\beta_0) \cup H(\beta_1)} \mathbb{E}_f(U_{n,\hat{j}}-\phi(f))^2 \lesssim r_n(\beta)$$
 %where, 
% \[
 %r_{n}\left( \beta \right) =\left\{ 
 %\begin{array}{c}
 %\log \left( n\right) ^{\frac{4\beta _{1}}{4\beta _{1}+1}}n^{-\frac{8\beta _{1}%
 %}{1+4\beta _{1}}}\text{ if }\beta =\beta _{1} \\ 
 %n^{-\frac{8\beta _{0}}{1+4\beta 0}}\text{ \ if }\beta =\beta _{0}%
% \end{array}%
 %\right. 
% \]
  We then have the following theorem. 
\begin{theorem}\label{theorem_two_point_quadratic}
For a deterministic constant $C$ depending on the parent wavelets, we have the following result. Let
$$\hat{j}=\I\left(\ihatk>C\sqrt{\log{n}}\frac{\sqrt{\kstar}}{n}\right)$$
and let $$\hat{\phi}=U_{n,\hat{j}}.$$ Then 
$$\sup_{f \in H(\beta_0)}\Ef\left[\hat{\phi}-\phi(f)\right]^2\lesssim n^{-\frac{8\beta_0}{1+4\beta_0}},$$
$$\sup_{f \in H(\beta_1)}\Ef\left[\hat{\phi}-\phi(f)\right]^2\lesssim (\log{n})^{\frac{4\beta_1}{1+4\beta_1}}n^{-\frac{8\beta_1}{1+4\beta_1}}.$$ 
\end{theorem}
The proof can be found in the appendix. However, it is worth mentioning that the two crucial ingredients of the proof is control of the error of testing at a suitably vanishing level and bounding higher order central moments of the non-adaptive minimax estimators dependent on non-random truncation level $k$. These properties are both derived from exponential tail inequalities for second order U-statistics using results of \cite{gine2000exponential,houdre2003exponential}. In particular, this is a typical structure for most applications of Lepski's method i.e. the treatment crucially relies on some exponential inequalities for the deviations of the sequence of non-adaptive minimax estimators \citep{birge2001alternative}. For estimation of general non-linear functionals this poses a challenge. The next section is devoted in understanding a possible way out.
%\begin{remark}

%\end{remark}

\section{A General Result by a Modification of Lepski's Method}\hspace*{\fill} \\
The results of the previous section strictly hold only for the quadratic functional $\phi(f)=\int f^2 d\mu$ so far. This is because, we have crucially used a deviation inequality which only holds for degenerate second order U-statistics. The minimax estimator for $\phi(f)=\int f^r  d\mu$ is in general a $r^{\mathrm{th}}$ order U-statistic. In order to apply standard Lepski's Method we will need to obtain suitable exponential deviation inequalities for higher order U-statistics. Although, such moment inequalities do exist \citep{adamczak2006moment}, the bounds include complicated quantities which needs to be controlled in a problem specific manner. However, owing to \cite{tchetgen2008minimax}, we observe that the biases of candidate estimators are within multiplicative constants of that of the estimator for quadratic functional and the variances remain of same order as well. Since we are interested in estimation in squared error norm, this motivates us to use the idea discussed in Section 2 to construct our final estimator. We now explain this below.

For a given choice of $d>1$, let $N$ be the largest integer such that $d^{N-1}\leq n^{1-\frac{2}{\log{\log{n}}}}$. Set $k_{j-1}=d^{j-1} n $ for $j=1,\ldots, N$. This holds if and only if, $(N-1)\leq \frac{\log{n}}{\log{d}}\left(1-\frac{2}{\log{\log{n}}}\right)$ and also that $\frac{k_N}{n^2}=o(\frac{1}{n})$. Now, for $j=0,\ldots,N-1$, define $\beta_j$ to be the solution of $k_j=n^{\frac{2}{1+4\beta_j}}$. This implies that $1+4\beta_j=\frac{2\log{n}}{(j-1)\log{d}+\log{n}},\ j=0,\ldots,N-1$. Also, note that by construction, $k_0\leq k_1 \leq \ldots \leq k_{N-1}$ and therefore, $\beta_0\geq \beta_1 \geq \ldots \geq \beta_{N-1}$. Also, a simple calculation shows that, $\beta_{N-1}\geq \frac{1}{4\log{\log{n}}-1}$. Therefore, $\frac{1}{4}\geq\beta_0\geq \beta_1 \geq \ldots \geq \beta_{N-1}\geq \frac{1}{4\log{\log{n}}-1}$. Further, note that for all $0\leq l_1\leq l_2 \leq N-1$, there exists constants $c_1,c_2$ such that $\beta_{l_1}-\beta_{l_2}\in \left[c_1\frac{l_2-l_1}{\log{n}}, c_2\frac{l_2-l_1}{\log{n}}\right]$. Further, define $k_j^*=\frac{k_j}{\log{n}^{\frac{1}{1+4\beta_j}}}=\left(\frac{n^2}{\log{n}}\right)^{\frac{1}{1+4\beta_j}}$ and $R(k_j^*)=\frac{k_j^*}{n^2}$. This definition implies that for $l_1 >l_2$, $\frac{k_{l_1}^*}{k_{l_2}^*}=\left(\frac{n^2}{\log{n}}\right)^{\frac{4\beta_{l_2}-4\beta_{l_1}}{(1+4\beta_{l_1})(1+4\beta_{l_2})}}\rightarrow \infty$. However, note that $\frac{k_{l_1}^*}{k_{l_2}}$ might not enjoy similar properties for certain ranges of $l_1>l_2$. Especially, for $\beta$ values within $\frac{1}{\log{n}}$ rate of each other, the corresponding $k$ values do not enjoy the above mentioned property. Let $s^*=s^*(n)$ be the smallest integer such that, $k(\beta_{s^*})^* \geq n$. Finally, define 
\be
\jhat&:=\min\left\{j: \hat{I}^2(k_j^*,k_l^*)\leq C_{\mathrm{opt}}^2 \log{n} R(k_l^*) \ \forall \ l\geq j, s^*\leq j \leq N-1 \right\}.
\ee
where $C_{\mathrm{opt}}$ is a deterministic constant to be specified later and similar to Section 3, $\hat{I}(k_j,k_l)=U_{k_l}-U_{k_j}$ for $k_j<k_l$.
\begin{theorem}\label{theorem_modified_lepski}
Suppose an estimator $U_{n,k}$ of $\phi(f)$ satisfies the following properties.
\begin{enumerate}
	\item For sufficiently large $k,n $ and fixed choice of $q\geq 1$,
	$$\sup_{f \in H(\beta)}\Ef^{\frac{1}{q}}\left[\left(U_{n,k}-\Ef(U_{n,k})\right)^{2q}\right]\leq C_q \frac{k}{n^2}$$  
	\item There exists a constant $C$ such that for any $f \in H(\beta,C)$, 
	$$|I_c(k_1,k_2)|\leq C I(k_1,k_2),$$ 
	for sufficiently large $k_1<k_2$ where $I_c(k_1,k_2)=\Ef(U_{n,k_2})-\Ef(U_{n,k_1})$ and by our previous convention $I(k_1,k_2)=\int f_{k_2}^2-\int f_{k_1}^2$.
\end{enumerate}	
Then for any $\beta \in (0,\frac{1}{4})$ and $\epsilon>0$, we have that 
$$\sup_{f \in H(\beta,C)} \Ef\left[\left(U_{n,k_{\jhat}^*}-\phi(f)\right)^2\right]\lesssim n^{-\frac{8\beta}{1+4\beta}}\log{n}^{\frac{4(\beta+\epsilon)}{1+4(\beta+\epsilon)}}.$$
\end{theorem}
Some remarks are in order about the implications of the theorem above. It says that if a sequence of estimators, having common index with the minimax estimator of the quadratic functional, has bias and higher order moments of the same order of magnitude as the quadratic functional estimator, then following the discussion in Section 1, the new sequence of estimators also shares the same rate of convergence (up to almost matching logarithmic factor) towards its corresponding functional simultaneously over the functional spaces which dictates the adaptation of the quadratic functional. The $\beta+\epsilon$ appearing in the logarithmic factor, is due using a different test statistic than $U_{n,k}$ as dictated by Lepski's Method. Establishing an equivalence in the order of the bias for these functionals requires fairly standard arguments \citep{tchetgen2008minimax}. The main challenge lies in obtaining suitable bounds on higher order moments of these estimators. This requires control of moments of higher order U-statistics based on orthogonal projection kernels. We crucially use the structure of the projection kernels based on compactly supported wavelet basis. Following ideas from \cite{rltv2015}, we implement a binning argument to keep track of membership of sampled observations in partitions of the sample space created by a particular resolution level of the wavelet expansion, and this turns out to be crucial in controlling the higher order moments of our estimator at the right level.

\section{Minimax Adaptive Estimation of $\int f^3 d\mu$}\hspace*{\fill} \\
As mentioned earlier, based on the technique of \cite{birge1995estimation} one can show that the minimax estimator of a general non-linear functional $\phi(f)=\int T(f)d\mu$ for smooth $T$ can be constructed by using ideas from linear, quadratic and cubic functionals of the density and appealing to a standard Taylor expansion argument.  While producing adaptive estimators of non-linear functionals as well, a similar strategy can be followed and it becomes crucial to understand the adaptive estimation of $\int f^3 d\mu$. In particular, a non-adaptive minimax rate optimal estimator of $\phi \left( f\right)=\int f^3 d\mu $ is given by
\be
	\ & \widehat{\phi }_{k_{1} ,k_{2}
		,k_{3}} =\mathbb{V}_{n}\left[ \widetilde{K}%
	_{k_{1} ,k_{2} ,k_{3} }\left( X_{i_{1}},X_{i_{2}},X_{i_{3}}\right) \right] \\
	&=\mathbb{V}_{n}\left[ \int K_{k_{1}}\left( x,X_{i_{1}}\right)
	K_{k_{1}}\left( x,X_{i_{2}}\right) K_{k_{1}}\left( x,X_{i_{3}}\right) dx%
	\right] \\
	&+3\mathbb{V}_{n}\left[ \int K_{k_{1}}\left( x,X_{i_{1}}\right) \left(
	K_{k_{3}}\left( x,X_{i_{2}}\right) -K_{k_{1}}\left( x,X_{i_{2}}\right)
	\right) K_{k_{1}}\left( x,X_{i_{3}}\right) dx\right] \\
	&+3\mathbb{V}_{n}\left[ \int K_{k_{1}}\left( x,X_{i_{1}}\right) \left(
	K_{k_{3}}\left( x,X_{i_{2}}\right) -K_{k_{1}}\left( x,X_{i_{2}}\right)
	\right) \left( K_{k_{3}}\left( x,X_{i_{3}}\right) -K_{k_{1}}\left(
	x,X_{i_{3}}\right) \right) dx\right] \\
	&+3\mathbb{V}_{n}\left[ \int \left( 
	\begin{array}{c}
		\left( K_{k_{2}}\left( x,X_{i_{1}}\right) -K_{k_{1}}\left(
		x,X_{i_{1}}\right) \right) \left( K_{k_{2}}\left( x,X_{i_{2}}\right)
		-K_{k_{1}}\left( x,X_{i_{2}}\right) \right) \\ 
		\left( K_{k_{3}}\left( x,X_{i_{3}}\right) -K_{k_{2}}\left(
		x,X_{i_{3}}\right) \right)%
	\end{array}%
	\right) dx\right] \\
	&+\mathbb{V}_{n}\left[ \int \left( 
	\begin{array}{c}
		\left( K_{k_{2}}\left( x,X_{i_{1}}\right) -K_{k_{1}}\left(
		x,X_{i_{1}}\right) \right) \left( K_{k_{2}}\left( x,X_{i_{2}}\right)
		-K_{k_{1}}\left( x,X_{i_{2}}\right) \right) \\ 
		\left( K_{k_{2}}\left( x,X_{i_{3}}\right) -K_{k_{1}}\left(
		x,X_{i_{3}}\right) \right)%
	\end{array}%
	\right) dx\right],
	\ee

where $k_{1}=k_{1}\left( \beta\right) \sim n,n^{\left( 3/2-2\beta \right) /\left(
	1+4\beta \right) }\leq k_{2}=k_{2}\left( \beta\right) \leq n^{\left( 3/2+2\beta
	\right) /\left( 1+4\beta \right) }$ and $k_{3}=k_{3}\left( \beta\right) \sim
n^{2/\left( 1+4\beta \right) }$ and $\mathbb{V}_n\left(h(X_{i_1},X_{i_2},X_{i_3})\right)$ corresponds to to the U-statistic based on $h$ i.e. $\mathbb{V}_n\left(h(X_{i_1},X_{i_2},X_{i_3})\right)={\nthree}\sum \limits_{i_1 \neq i_2 \neq i_3}{h(X_{i_1},X_{i_2},X_{i_3})}$; see \cite{tchetgen2008minimax} for more details. By \cite{tchetgen2008minimax}, the bias of this 
estimator is given by 

\be
\ & \Ef\left(\widehat{\phi }_{k_{1}\left( \beta \right) ,k_{2}\left( \beta \right)
		,k_{3}\left( \beta \right) }\right)-\phi(f)\\&=\int \left( f_{k_{3}}-f_{k_{2}}\right) ^{3}+3\int
\left( f_{k_{3}}-f_{k_{2}}\right) ^{2}\left( f_{k_{2}}-f_{k_{1}}\right)
+\int \left( f^{3}-f_{k_{3}}^{3}\right) \\
\label{eqn:bias_third_order}
\ee 
which is dominated by $$\int \left(
f^{3}-f_{k_{3}}^{3}\right) \sim k_{3}^{-2\beta }\leq C\frac{k_3}{n^2}.$$ 
Now note that, 

\be
\sup_{f \in H(\beta,C)}|\Ef(\widehat{\phi }_{k^{(2)}_{1} ,k^{(2)}_{2}
				,k^{(2)}_{3}}-\Ef(\widehat{\phi }_{k^{(1)}_{1} ,k^{(1)}_{2}
		,k^{(1)}_{3}}|&\sim \left(k_3^{(1)}\right)^{-2\beta}, \label{eqn:comparative_bias_third_order}
		\ee whenever $k^{(1)}=\left(k_1^{(1)},k_2^{(1)},k_3^{(1)}\right)\leq k^{(2)}=\left(k_1^{(2)},k_2^{(2)},k_3^{(2)}\right)$ where we say $v_1 \leq v_2$ for two vectors $v_1,v_2 \in \mathbb{R}_+^s$ whenever $\max v_1 \leq \max v_2$. 
As shown by \cite{tchetgen2008minimax}, the variance is also dominated by $\frac{k_3}{n^2}$. Therefore, similar to the discussion in Section 2, the bias and variance of the candidate estimator is similar to that of the estimator of the quadratic functional of the density. We will use this fact to produce an adaptive estimator. The next result is the main result of this section and helps us apply Theorem \ref{theorem_modified_lepski} to construct an adaptive estimator of $\phi(f)=\int f^3 d\mu$.
\begin{theorem}\label{theorem_fcube_moment}
For any $k_1\leq k_2 \leq k_3$ as above, one has for all sufficiently large $n$,
$$\Ef\left(\widehat{\phi }_{k_{1} ,k_{2},k_{3}}-\Ef\left(\widehat{\phi }_{k_{1} ,k_{2},k_{3}}\right)\right)^{2q} \leq C(\psimax,\fmax,\fmin)\left(\frac{k_3}{n^2}\right)^q,$$
whenever the projection kernels in the construction are based on the Haar Basis.
\end{theorem}
Now combining \eqref{eqn:bias_third_order}, \eqref{eqn:comparative_bias_third_order}, Theorem \ref{theorem_fcube_moment} with Theorem \ref{theorem_modified_lepski}, we can construct an adaptive estimator of cubic integral functional as follows. \\

 We use the ideas laid down heuristically in Section 2 and detailed in Section 4. In particular, fix a $d>1$. Let $N$ be the largest integer such that $d^{N-1}\leq n^{1-\frac{2}{\log{\log{n}}}}$. Set $l_{j-1}=d^{j-1} n $ for $j=1,\ldots, N$. Now, for $j=0,\ldots,N-1$, define $\beta_j$ to be the solution of $l_j=n^{\frac{2}{1+4\beta_j}}$. Further, define $l_j^*=\frac{l_j}{\log{n}^{\frac{1}{1+4\beta_j}}}=\left(\frac{n^2}{\log{n}}\right)^{\frac{1}{1+4\beta_j}}$ and $R(l_j^*)=\frac{l_j^*}{n^2}$. Let $s^*=s^*(n)$ be the smallest integer such that, $l(\beta_{s^*})^* \geq n$. Finally, define 
 \be
 \jhat&:=\min\left\{j: \hat{I}^2(l_j^*,l_m^*)\leq C_{\mathrm{opt}}^2 \log{n} R(k_m^*) \ \forall \ m\geq j, s^*\geq j \leq N-1 \right\}.
 \ee
 where $C_{\mathrm{opt}}$ will be a deterministic constant and similar to Section 3, $\hat{I}(l_j,l_m)=U_{l_m}-U_{l_j}$ for $l_j<l_m$. Now let
 $$\hat{\phi}=\hat{\phi}_{\hat{k_1},\hat{k_2},\hat{k_3}},$$
 where $\hat{k_1} \sim n$, $\hat{k_2} \sim n^{\frac{\frac{3}{2}+2\beta_{\jhat}}{1+4\beta_{\jhat}}}$ and $\hat{k_3} \sim l_{\jhat}^*$.
\begin{theorem}\label{theorem_adptation_cubic}
Consider projection kernels based on Haar wavelets. The estimator $\hat{\phi}$ of $\phi(f)=\int f^3$ satisfies for any $\beta \in (0,\frac{1}{4})$ and $\epsilon>0$, 
$$\sup_{f \in H(\beta,C)} \Ef\left[\left(\hat{\phi}-\phi(f)\right)^2\right]\lesssim n^{-\frac{8\beta}{1+4\beta}}\log{n}^{\frac{4(\beta+\epsilon)}{1+4(\beta+\epsilon)}},$$
for some deterministic $C_{\mathrm{opt}}$ (depending only on parent wavelets and $\fmax$).
\end{theorem}
%Now define $$I\left( k\left( \beta _{0}\right) ,k\left( \beta _{1}\right) \right)
%=\psi _{k_{1}\left( \beta _{1}\right) ,k_{2}\left( \beta _{1}\right)
	%,k_{3}\left( \beta _{1}\right) }-\psi _{k_{1}\left( \beta _{0}\right)
%	,k_{2}\left( \beta _{0}\right) ,k_{3}\left( \beta _{0}\right) },$$
	%where 
	
%\be
%\psi _{k_{1},k_{2},k_{3}}&=\int f_{k_{1}}^{3}+3\int \left(
%f_{k_{3}}-f_{k_{1}}\right) ^{2}f_{k_{1}}+3\int \left(
%f_{k_{3}}-f_{k_{1}}\right) f_{k_{1}}^{2}\\
%&+\int \left( f_{k_{2}}-f_{k_{1}}\right) ^{3}+3\int \left(
%f_{k_{2}}-f_{k_{1}}\right) ^{2}\left( f_{k_{3}}-f_{k_{2}}\right) ,
%\ee
%and  $k\left(
%\beta \right) =\left( k_{1}\left( \beta \right) ,k_{2}\left( \beta \right)
%,k_{3}\left( \beta \right) \right) $. 
\begin{remark}
For the proof of Theorem \ref{theorem_fcube_moment}, the requirement of Haar wavelet arises at one specific instance in the proof. We expect that similar results continue to hold for more general compactly supported wavelets using arguments developed herein. However we do not further consider other wavelet bases. 
\end{remark}

\section{Minimax Adaptive Estimation of $\int T(f) d\mu$}\hspace*{\fill} \\
We now discuss adaptive estimation of a more general integral functional of the density. As mentioned earlier that such functionals arise naturally in information theory. In particular, a large body of work has focused on estimating the Shannon entropy \citep{beirlant1997nonparametric} and more recent works include estimation of Renyi and Tsallis entropies \citep{leonenko2010statistical,pal2010estimation}. As mentioned, we will use ideas developed so far for adaptive estimation of up to cubic functionals to come up with an adaptive estimator of smooth non-linear functionals. We only sketch the idea here and omit technical details. In particular, suppose that $T \left( f\right) $ admits an expansion around $f_{0}$ such that
\be T \left( f\right) \left( x\right)&=T \left( f_{0}\right) \left(
x\right) +T^{\left( 1\right) }\left( f_{0}\right) \left( x\right) \left(
f\left( x\right) -f_{0}\left( x\right) \right)+T ^{\left( 2\right)
}\left( f_{0}\right) \left( f\left( x\right) -f_{0}\left( x\right) \right)
^{2}\\
&+T ^{\left( 3\right) }\left( f_{0}\right) \left( f\left( x\right)-f_{0}\left( x\right) \right) ^{3}+O\left( \left( f\left( x\right)
-f_{0}\left( x\right) \right) ^{4}\right)
\ee as $f\left( x\right) \rightarrow
f_{0}\left( x\right) .$ The idea now is to sample split and obtain an adaptive estimator $\widehat{f}\left( x\right) $ of $f\left( x\right) $ so that $\left( \widehat{f}\left( x\right)
-f\left( x\right) \right) =O_{p}\left( n^{-2\beta /\left( 1+2\beta \right)}\right) $. Next,
\be
\phi \left( f\right) &=\int T \left( \widehat{f}\right) \left(
x\right) +\int T ^{\left( 1\right) }\left( \widehat{f}\right) \left(
x\right) \left( f\left( x\right) -\widehat{f}\left( x\right) \right) +\int
 T^{\left( 2\right) }\left( \widehat{f}\right) \left( f\left( x\right) -
\widehat{f}\left( x\right) \right) ^{2}\\
&+\int T ^{\left( 3\right) }\left( \widehat{f}\right) \left( f\left(
x\right) -\widehat{f}\left( x\right) \right) ^{3}+O\left( \int \left(
f\left( x\right) -\widehat{f}\left( x\right) \right) ^{4}\right) .
\ee
Note that, $\int \left( f\left( x\right) -\widehat{f}\left( x\right) \right)
^{4}=O_{P}\left( n^{-8\beta /\left( 1+2\beta \right) }\right) =o_p\left(n^{-8\beta
/\left( 1+4\beta \right) }\right)$. Therefore, we need only learn how to adapt to functionals of the form $\int
g_{1}\left( \widehat{f}\right) f\left( x\right) +\int g_{2}\left( \widehat{f}%
\right) f\left( x\right) ^{2}+\int g_{3}\left( \widehat{f}\right) f\left(
x\right) ^{3}.$
The linear functional estimation theory is well understood and the quadratic or cubic terms
are rather straightforward generalizations of our statistics. For the
quadratic term for instance, use $\mathbb{V}_{n}\left[ \int g_{2}\left( 
\widehat{f}\right) \left( x\right) K_k\left( x,X_{i_{1}}\right) K_k\left(
x,X_{i_{2}}\right) dx\right] $ instead of the one used earlier. 
\section{Adaptation Lower bound for Estimation of $\int T(f) d\mu$}\hspace*{\fill} \\
In this section, we implement ideas similar to \cite{efromovich1996bickel} to provide a lower bound on the required price to be paid for adaptation over $\beta<\frac{1}{4}$. \cite{efromovich1996bickel} proved that, while estimating a quadratic functional of the density, if an estimator achieves a parametric rate for $\beta\geq \frac{1}{4}$ then it must incur a heavier penalty than $(\log{n})^{\frac{4\beta}{1+4\beta}}$ for $\beta<\frac{1}{4}$. Here we provide a result of similar flavor regarding constraints on estimation rates at two points $\beta_1,\beta_2 <\frac{1}{4}$. 

We begin by describing the main tool in our proof, which is a general version of constrained risk inequality due to \cite{cai2011testing}, obtained as an extension of \cite{brown1996constrained}. For the sake of completeness, begin with a Summary of these results. Suppose $Z$ has distribution $\mathbb{P}_{\theta}$ where $\theta$ belongs to some parameter space $\Theta$. Let $\hat{Q}=\hat{Q}(Z)$ be an estimator of a function $Q(\theta)$ based on $Z$ with bias $B(\theta):=\mathbb{E}_{\theta}(\hat{Q})-Q(\theta)$. Now suppose that $\Theta_0$ and $\Theta_1$ form a disjoint partition of $\Theta$ with priors $\pi_0$ and $\pi_1$ supported on them respectively. Also, let $\mu_i=\int Q(\theta)d\pi_i$ and $\sigma_i^2=\int (Q(\theta)-\mu_i)^2d\pi_i$, $i=0,1$ be the mean and variance of $Q(\theta)$ under the two priors $\pi_0$ and $\pi_1$. Letting $\gamma_i$ be the marginal density with respect to some common dominating measure of $Z$ under $\pi_i$, $i=0,1$, let us denote by $\mathbb{E}_{\gamma_0}(g(Z))$ the expectation of $g(Z)$ with respect to the marginal density of $Z$ under prior $\pi_0$ and distinguish it from $\mathbb{E}_{\theta}(g(Z))$, which is the expectation under $\mathbb{P}_{\theta}$. Lastly, denote the chi-square divergence between $\gamma_0$ and $\gamma_1$ by $\chi=\left\{\mathbb{E}_{\gamma_0}\left(\frac{\gamma_1}{\gamma_0}-1\right)^2\right\}^{\frac{1}{2}}$. Then we have the following result.
\begin{prop}[\cite{cai2011testing}]\label{prop_cai}
If $\int \mathbb{E}_{\theta}\left(\hat{Q}(Z)-Q(\theta)\right)^2d\pi_0(\theta)\leq \epsilon^2$, then
$$\vert\int B(\theta)d\pi_1(\theta)-\int B(\theta)d\pi_0(\theta)\vert\geq |\mu_1-\mu_0|-(\epsilon+\sigma_0)\chi.$$
%In particular, 
%$$\max_{i=0,1}\int \mathbb{E}_{\theta}\left(\hat{Q}(Z)-Q(\theta)\right)d\pi_i(\theta)\geq \frac{(|\mu_1-\mu_0|-\sigma_0 \chi)^2}{(\chi+2)^2}$$
\end{prop}
Since the maximum risk is always at least as large as the average risk, this immediately yields a lower bound on the minimax risk. In order to use this result, we will need to produce suitable priors on appropriate parameter spaces, so that we capture the price needed for adaptation. Since the uniform density $f_0\equiv \I(0,1)$ is in all the smoothness classes we consider here, we take $\Theta_0=\{f_0^{(n)}\}$ and the prior $\pi_0=\delta_{f_0^{(n)}}$, the dirac mass at the uniform density. Here and below $f^{(n)}$ refers to the joint density of $X_1,\ldots,X_n$ which are i.i.d with density $f$. Now let us construct the alternative parameter space along the lines of \cite{efromovich1996bickel}, which in turn relies on \cite{ingster1987minimax}. Take $h$ to be a function supported on $[0,1]$ such that $\int h =0$, $\int h^2=c$ for some $c>0$ to be specified later, $h \in H(1,C)$, and $1+h \geq 0$. Let $v_n$ be an increasing sequence of positive integers, to be specified later up to multiplicative constants, and denote by $\overline{a}$ the vector $(a_0,\ldots,a_{v_n-1})\in \{-1,+1\}^{v_n}$. Now define $$f_{\overline{a}}=f_0+\sum_{i=0}^{v_n-1}a_i v_n^{-\beta}h(v_n x-i),$$ and let $\Theta_1=\{f^{(n)}_{\overline{a}}: \overline{a}\in \{-1,+1\}^{v_n}\}.$ First,note that by construction one always has $f_{\overline{a}} \in H(\beta,C)$. Finally let $\pi_1$ be the uniform prior putting $\frac{1}{2^{v_n}}$ mass at each point of $\Theta_1$. In order to apply Proposition \ref{prop_cai}, we evaluate the following quantities. Using the notation introduced earlier, $$\mu_0=\int T(f_0(x))dx=T(1)$$ and $\sigma_0=0$. Also,
 $$\mu_1=\frac{1}{2^{v_n}}\mathlarger{\mathlarger{\sum}}\limits_{\overline{a}\in \{-1,+1\}^{v_n}}\int T\left(f_0(x)+\sum_{i=0}^{v_n-1}a_i v_n^{-\beta}h(v_n x-i)\right)dx.$$ Therefore $$|\mu_1-\mu_0|=\frac{1}{2^{v_n}}\mathlarger{\mathlarger{\sum}}\limits_{\overline{a}\in \{-1,+1\}^{v_n}}\int \left[T\left(1+\sum_{i=0}^{v_n-1}a_i v_n^{-\beta}h(v_n x-i)\right)-T(1)\right]dx.$$ Now by Taylor expansion, for each $\overline{a}$, 
\be
 \ & \int \left[T\left(1+\sum_{i=0}^{v_n-1}a_i v_n^{-\beta}h(v_n x-i)\right)-T(1)\right]\\
 &=T'(1)\sum_{i=0}^{v_n-1}a_i v_n^{-\beta}\int h(v_n x-i)dx+T''(1)\int \left[\sum_{i=0}^{v_n-1}a_i v_n^{-\beta} h(v_n x-i)\right]^2 dx \\
 &+ \int T'''(\xi(x))\left[\sum_{i=0}^{v_n-1}a_i v_n^{-\beta} h(v_n x-i)\right]^3 dx  \ \text{where}\ {\|\xi-1\|_{\infty}\leq 1}\\ 
&\geq C_T v_n^{-2\beta}
\ee
for some constant $C_T>0$ depending only on $T$. The above calculations follow by the inclusion of the support of $h$ in $[0,1]$, boundedness of $h$ and standard change of variables. Therefore, $|\mu_1-\mu_0| \geq C v_n^{-2\beta}$ as well, where we have dropped the subscript $T$ for notational convenience. Now suppose that the bias at $f_0$ is smaller than $\epsilon_n \lesssim \left(\frac{n}{\sqrt{\log{n}}}\right)^{-\frac{4\beta^*}{4\beta^*+1}}$ for some $\beta^* <\frac{1}{4}$. Then for any for any $\beta<\beta^*$, we have by proportion \ref{prop_cai} that $\int B(\theta)d\pi_1(\theta)\geq C v_n^{-2\beta}-\epsilon_n \chi$. Therefore, if we can show that for $v_n = \left(\frac{n}{\sqrt{c\log{n}}}\right)^{\frac{2}{1+4\beta}}$ one has $\chi\ll n^{c}$, then one gets the desired rate for the incurred penalty by choosing $c$ small enough. This can be derived following the arguments in \cite{ingster1987minimax}. Therefore, we have sketched the proof of the following theorem.
\begin{theorem}\label{theorem_lower_bound}
	Suppose  one has 
	$$\sup_{f \in H(\beta^*)}\Ef\left(\hat{\phi}-\phi(f)\right)^2 \lesssim \left(\frac{n^2}{\log{n}}\right)^{-\frac{4\beta^*}{1+4\beta^*}},$$
	for an estimator $\hat{\phi}$ of $\phi(f)=\int T(f)d\mu$ with $T$ having bounded third derivative. Then for any $\beta>\beta^*$, 
	 $$\sup_{f \in H(\beta)}\Ef\left(\hat{\phi}-\phi(f)\right)^2 \gtrsim \left(\frac{n^2}{\log{n}}\right)^{-\frac{4\beta}{1+4\beta}}.$$
\end{theorem}

\section{Discussions}
Our results are all provided for one dimension $d=1$ corresponding to a low smoothness regime ${\beta} < \frac{d}{4}$. Although we do not provide explicit details, our proofs can be easily extended to incorporate the case of $d>1$ and to consider $\sqrt{n}$ rate of convergence of the constructed estimators for $\beta>\frac{d}{4}$. The case of $\beta=\frac{d}{4}$ is a little more subtle; however this also can be done by adding one more level of discretization near the truncation level $k=n$. Since, the crux of the arguments remain the same, we do not elaborate on such proofs here.

Following the idea of possibly modifying Lepski's method to include a larger class of mechanisms to choose a tuning parameter for adaptive estimation, another potentially interesting question is to investigate which method provides a better performance- either in a finite sample sense or regarding constants of asymptotic mean squared error. Finally, the study of adaptive non-parametric divergence and entropy based on more than one densities is also of interest. Since, the quadratic functional estimator is a U-statistics is based on a bounded kernel, the results of \cite{houdre2003exponential} were readily applicable for deriving a suitable exponential tail bound. It is worth exploring how to modify such arguments for second order U-statistics estimators that arise in context of non-parametric regression problems and fail to be bounded in case of unbounded error distributions. We keep the study of such questions for possible future research.

\appendix
\section{Proof of Theorems}
\subsection*{Proof of Theorem \ref{theorem_two_point_quadratic}}
\begin{proof}
	Recall that,
	\begin{eqnarray*}
		&&\sup_{f}E\left[ U_{n,\widehat{j}}-\phi \left( f\right) \right] ^{2} \\
		&\leq &\sum\limits_{k=0}^{1}\sup_{f}E\left\{ \I\left\{ \widehat{j}%
		=k\right\} \left[ U_{n,\widehat{j}}-\phi \left( f\right) \right] ^{2}\right\}
		\\
		&\equiv & \sup_f R_{0}(\beta_f)+\sup_f R_{1}(\beta_f)
	\end{eqnarray*}
	
	\subsubsection*{Higher Smoothness {$(\beta_f=\beta_0)$}}\hspace*{\fill} \\
	(i) First let us consider,
	\be 
	\sup_f R_{0}(\beta_0)&=\sup_{f \in H(\beta_0)}\Ef\left[\I(\jhat=0)\left(U_{n,0}-\phi(f)\right)^2\right]\\
	&\leq \sup_{f \in H(\beta_0)}\Ef\left[\left(U_{n,0}-\phi(f)\right)^2\right]\lesssim n^{-\frac{8\beta_0}{1+4\beta_0}}
	\ee
	by our choice of $U_{n,0}$.
	
	(ii) Next we consider $\sup_{f \in H(\beta_0)}R_1(\beta_0)$. For this part, we will bound $R_1(\beta_0)$ from above uniformly in all $f \in H(\beta_0)$. First note that by Holder's Inequality, for any conjugate pair $(p,q)$ we have 
	\be 
	R_1(\beta_0)&=\Ef\left[\I(\jhat=1)\left(U_{n,1}-\phi(f)\right)^2\right]\\
	& \leq \left\{\Pf\left(\I(\jhat=1)\right)\right\}^{\frac{1}{p}}\left\{\Ef\left[\left(U_{n,1}-\phi(f)\right)^{2q}\right]\right\}^{\frac{1}{q}}
	\ee
	The proof of $\sup_{f \in H(\beta_0)}R_1(\beta_0) \lesssim n^{-\frac{8\beta_0}{1+4\beta_0}}$ now follows by the following two lemmas.
	\begin{lemma}\label{lemma_wrong_jhat}
		For kernels based on compactly supported wavelet bases,
		$$\sup_{f \in H(\beta_0)}\Pf\left(\I(\jhat=1)\right) \lesssim \frac{1}{n}$$.
	\end{lemma}
	
	\begin{lemma}\label{lemma_error_lqnorm}
		For kernels based on compactly supported wavelet bases,
		$$\sup_{f \in H(\beta_0)}\left\{\Ef\left[\left(U_{n,1}-\phi(f)\right)^{2q}\right]\right\}^{\frac{1}{q}} \lesssim \frac{k_1}{n^2}$$.
	\end{lemma}
	
	\subsubsection*{Lower Smoothness {$(\beta_f=\beta_1)$}}\hspace*{\fill} \\
	
	(i) First note that,
	\be 
	\sup_f R_{1}(\beta_1)&=\sup_{f \in H(\beta_1)}\Ef\left[\I(\jhat=1)\left(U_{n,1}-\phi(f)\right)^2\right]\\
	&\leq \sup_{f \in H(\beta_1)}\Ef\left[\left(U_{n,1}-\phi(f)\right)^2\right]\lesssim n^{-\frac{8\beta_1}{1+4\beta_1}}
	\ee
	by our choice of $U_{n,1}$.\\

	(ii) Now,
	\be
	R_{0}(\beta_1)=\Ef\left( \mathbf{1}\left\{ \widehat{I}^{2}\left( k_{0},k_{1}^{\ast
	}\right) <C\ln \left( n\right) \frac{k_{1}^{\ast }}{n^{2}}\right\} \left[
	U_{n,0}-\phi \left( f\right) \right] ^{2}\right).
	\ee
	We consider three cases as below with $\nu _{n}=\ln
	(n)^{-1/\left( 8\beta _{1}+2\right) }$. In the calculations below $C, C^{'}, C^{''}, C^{'''}$ are arbitrary constants which can be chosen for the calculations to be valid and can change from place to place. Also for controlling higher order moments of related second order U-statistics below, we use results along the lines of Lemma \ref{lemma_error_lqnorm} with obvious modifications to the proof.\\
	
		\textbf{Case 1:} $\mathbf{I^{2}\left( k_{0},k_{1}^{\ast }\right) >C^{'}\left( 1+\nu _{n}\right)
		\ln \left( n\right) \frac{k_{1}^{\ast }}{n^{2}}}$ \\
	
		In this case we have,
		\be
		\ & R_{0}(\beta)\\
		&= \Ef\left(  
		\mathbf{1}\left\{ \widehat{I}^{2}\left( k_{0},k_{1}^{\ast }\right) <C\ln
		\left( n\right) \frac{k_{1}^{\ast }}{n^{2}}\right\} 
		\mathbf{1}\left\{ I^{2}\left( k_{0},k_{1}^{\ast }\right) >C^{'}\left( 1+\nu
		_{n}\right) \ln \left( n\right) \frac{k_{1}^{\ast }}{n^{2}}\right\} \left[
		U_{n,0}-\phi \left( f\right) \right] ^{2}\right) \\		
		&\leq  \Ef\left(  \left( 
		\mathbf{1}\left\{ I\left( k_{0},k_{1}^{\ast }\right) -\widehat{I}\left(
		k_{0},k_{1}^{\ast }\right) >\left( C^{''}\nu _{n}\right) I\left(
		k_{0},k_{1}^{\ast }\right) \right\} 
		\left[ U_{n,0}-\phi \left( f\right) \right] ^{2}\right) \right) \\
		&\leq C^{'''}\left( \nu _{n}I\left( k_{0},k_{1}^{\ast }\right) \right)
		^{-2}\Ef\left( \left( I\left( k_{0},k_{1}^{\ast }\right) -\widehat{I}\left(
		k_{0},k_{1}^{\ast }\right) \right) ^{2}\left[ U_{n,0}-\phi \left( f\right) 
		\right] ^{2}\right) \\
		&\leq C^{'''}\left( \nu _{n}I\left( k_{0},k_{1}^{\ast }\right) \right) ^{-2}
		\Ef^{1/2}\left( \left( I\left( k_{0},k_{1}^{\ast }\right) -\widehat{I}\left(
		k_{0},k_{1}^{\ast }\right) \right) ^{4}\right) \Ef^{1/2}\left( \left[
		U_{n,0}-\phi \left( f\right) \right] ^{4}\right) \\
		& \leq C^{'''}\left( \nu _{n}I\left( k_{0},k_{1}^{\ast }\right) \right) ^{-2}\left(
		n^{-2}k_{1}^{\ast }\right) \left( I\left( k_{0},k_{1}^{\ast }\right) +\left(
		k_{1}^{\ast }\right) ^{-2\beta _{1}}\right) ^{2}\leq C^{'''}\left( \nu _{n}\right) ^{-2}\left( n^{-2}k_{1}^{\ast }\right) \\
		&=C^{'''}\ln (n)^{1/\left( 4\beta _{1}+1\right) }\left( \frac{n^{2}}{\ln \left(
			n\right) }\right) ^{1/\left( 1+4\beta _{1}\right) }n^{-2}=C^{'''}n^{2/\left( 1+4\beta _{1}\right) }n^{-2}=C^{'''}n^{-\frac{8\beta_1}{1+4\beta_1}}
		\ee
		
			\textbf{Case 2:} $\mathbf{I^{2}\left( k_{0},k_{1}^{\ast }\right)\leq C\frac{k_1}{n^2}} $\\
		
	      In this case, 
	       $$\Ef\left( \left[ U_{n,0}-\phi \left( f\right) \right] ^{2}\right)
		=k_{0}^{-4\beta _{1}}+\frac{k_{0}}{n^{2}}\asymp I^{2}\left(
		k_{0},k_{1}\right) +\frac{k_0}{n^2} \leq C^{'}n^{-\frac{8\beta_1}{1+4\beta_1}}. $$
		
			\textbf{Case 3:} $\mathbf{I^{2}\left( k_{0},k_{1}^{\ast }\right) \leq C\left( 1+\nu _{n}\right)
				\ln \left( n\right) \frac{k_{1}^{\ast }}{n^{2}}}$ \textbf{and}
			$\mathbf{I^{2}\left( k_{0},k_{1}^{\ast }\right)\leq C\frac{k_1}{n^2}} $\\

			  In this case, $I^{2}\left( k_{0},k_{1}^{\ast }\right) <\left( 1+\nu
		_{n}\right) \ln \left( n\right) \frac{k_{1}^{\ast }}{n^{2}}$ and $%
		I^{2}\left( k_{0},k_{1}\right) >\frac{k_{1}^{\ast }}{n^{2}} $, 
		so that 
		\be 
		 \Ef\left( \left[ U_{n,0}-\phi \left( f\right) \right] ^{2}\right)	&\leq C^{'}\left( I^{2}\left( k_{0},k_{1}^{\ast }\right) +\left( k_{1}^{\ast
		}\right) ^{-4\beta _{1}}+\frac{k_{0}}{n^{2}}\right)  C^{''}\left( I^{2}\left(
		k_{0},k_{1}^{\ast }\right) +\left( k_{1}^{\ast }\right) ^{-4\beta _{1}}+%
		\frac{k_{0}}{n^{2}}\right) \\
		&\leq C^{'''}\ln \left( n\right) ^{1-1/\left( 1+4\beta _{1}\right) }n^{-8\beta
			_{1}/\left( 1+4\beta _{1}\right) }=C^{'''}n^{-8\beta _{1}/\left( 1+4\beta
			_{1}\right) }\ln \left( n\right) ^{4\beta _{1}/\left( 1+4\beta _{1}\right)
		}.
		\ee

\end{proof}

\subsection*{Proof of Theorem \ref{theorem_modified_lepski}}
\begin{proof}
	Suppose $\beta \in (\beta_{j+1},\beta_j]$.
	% Then we have the following as before,
	%\be
	%\Ef\left[\left(U_{n,k_{\jhat}}-\phi(f)\right)^2\right]&=\sum_{l=s^*}^{N-1}\Ef\left[\I(\jhat=l)\left(U_{n,k_{l}}-\phi(f)\right)^2\right]
	%\ee
	Let $l_c(C^*)$ be the largest $l$ such that $I^2(k_l^*,k_j^*)>C^* \log{n} R(k_j^*)$. Let us explain the reason for existence of such an $l_c(C^*)$. Note that it is enough to show that the ratio $\frac{I^2(k_l^*,k_j^*)}{C^* \log{n} R(k_j^*)}$ is upper bounded by a quantity that $l$ increases through positive integers. By our choice of $\beta \in (\beta_{j+1},\beta_j]$, we have that $\frac{I^2(k_l^*,k_j^*)}{C^* \log{n} R(k_j^*)}$ is at most a constant times $\frac{(k_l^*)^{-4\beta_{j+1}}}{\frac{k_j^*}{n^2}}$. The power of $n$ in this ration is $\frac{8\beta_{j}}{1+4\beta_j}-\frac{8\beta_{j+1}}{1+4\beta_l}$ which indeed decreases as $l$ increases. Also, note that trivially $l_c(C^*)<j$. Also, by definition of $l_c(C^*)$, for any $l\leq l_c$, $I^2(k_l^*,k_j^*)>C^* \log{n} R(k_j^*)$ and for $l>l_c(C^*)$, $I^2(k_l^*,k_j^*)\leq C^* \log{n} R(k_j^*)$. Our proof relies on the fact that the testing procedure does not select an index $l \leq l_c(C^*)$ or $l>j+1$ with high probability. This is captured by the following two lemmas.
	\begin{lemma}\label{lemma_l_less_than_lc}
		For kernels based on compactly supported wavelet bases, $$\sup_{f \in H(\beta)}\Pf\left(\jhat =l\right)\leq \frac{\const}{n},$$
		for any $l \leq l_c(C^*)$ whenever $C^*>4C_{\mathrm{opt}}^2$.
	\end{lemma}
	
	\begin{lemma}\label{lemma_l_bigger_than_jplusone}
		For kernels based on compactly supported wavelet bases, $$\sup_{f \in H(\beta)}\Pf\left(\jhat \geq j+2\right)\leq \frac{\const}{n},$$
		whenever $C^*>4C_{\mathrm{opt}}^2$.
	\end{lemma}
	Let us first complete the proof of Theorem \ref{theorem_modified_lepski} assuming the validity of Lemmas \ref{lemma_l_less_than_lc} and \ref{lemma_l_bigger_than_jplusone}. Note that,
	\be
	\Ef\left[\left(U_{n,k_{\jhat}^*}-\phi(f)\right)^2\right]&=\sum_{l=s^*}^{N-1}\Ef\left[\I(\jhat=l)\left(U_{n,k_{l}^*}-\phi(f)\right)^2\right]\\
	&=T_1+T_2+T_3, \label{eqn:main_modified_lepski}
	\ee
	where $T_1=\sum\limits_{l=s^*}^{l_c(C^*)}\Ef\left[\I(\jhat=l)\left(U_{n,k_{l}^*}-\phi(f)\right)^2\right]$, $T_2=\sum\limits_{l=l_c(C^*)+1}^{j+1}\Ef\left[\I(\jhat=l)\left(U_{n,k_{l}^*}-\phi(f)\right)^2\right]$ and $T_3= \sum\limits_{l=j+2}^{N-1}\Ef\left[\I(\jhat=l)\left(U_{n,k_{l}^*}-\phi(f)\right)^2\right] $. In the following we control the terms $T_h$, $h=1,2,3$ individually to show the desired result modulo the proofs of the above two lemmas. We then finish the proof by proving the lemmas. For the following let $(p,q)$ denote a pair of positive real numbers such that $\frac{1}{p}+\frac{1}{q}=1$. The specific choices of the pair will be clear from the proof. In particular, we will always choose $q$ to be an integer and $p$ will be sufficiently close to $1$.
	\subsection*{Control of $T_1$}
	\be
	T_1&=\sum\limits_{l=s^*}^{l_c(C^*)}\Ef\left[\I(\jhat=l)\left(U_{n,k_{l}^*}-\phi(f)\right)^2\right]\\
	& \leq \sum\limits_{l=s^*}^{l_c(C^*)}\Pf^{\frac{1}{p}}\left(\jhat=l\right)\Ef^{\frac{1}{q}}\left[\left(U_{n,k_{l}^*}-\phi(f)\right)^{2q}\right]\\
	& \leq \sum\limits_{l=s^*}^{l_c(C^*)} \left(\frac{\const}{n}\right)^{\frac{1}{p}}\left[\frac{k_l^*}{n^2}+(k_l^*)^{-2\beta}\right]\leq  \log{n}\left(\frac{\const}{n}\right)^{\frac{1}{p}}
	\ee
	where the second last inequality above follows from Lemma \ref{lemma_l_less_than_lc} and conditions of the theorem, and the last inequality follows from the choice of $N\lesssim \log{n}$. Therefore for $p$ sufficiently close to $1$, we have desired control over $T_1$. 
	\subsection*{Control of $T_2$}
	\be
	  T_2&=\sum\limits_{l=l_c(C^*)+1}^{j+1}\Ef\left[\I(\jhat=l)\left(U_{n,k_{l}^*}-\phi(f)\right)^2\right]\leq \sum\limits_{l=l_c(C^*)+1}^{j+1} \Pf^{\frac{1}{p}}\left(\jhat=l\right)\Ef^{\frac{1}{q}}\left[\left(U_{n,k_{l}^*}-\phi(f)\right)^{2q}\right]\\
	&\leq C_q \sum\limits_{l=l_c(C^*)+1}^{j+1} \Pf^{\frac{1}{p}}\left(\jhat=l\right)\left\{
	\begin{array}{c}
	\Ef^{\frac{1}{q}}\left[\left(U_{n,k_{l}^*}-\Ef\left(U_{n,k_{l}^*}\right)\right)^{2q}\right]
	\\+\Ef^{\frac{1}{q}}\left[\left(\Ef\left(U_{n,k_{l}^*}\right)-\Ef\left(U_{n,k_{j}^*}\right)\right)^{2q}\right]+\Ef^{\frac{1}{q}}\left[\left(U_{n,k_{j}^*}-\phi(f)\right)^{2q}\right]
	\end{array}\right\}\\
	&\leq C_q \sum\limits_{l=l_c(C^*)+1}^{j+1} \Pf^{\frac{1}{p}}\left(\jhat=l\right)\left\{\frac{k_l^*}{n^2}+I_c^2(k_j^*,k_l^*)+\frac{k_j^*}{n^2}+(k_j^*)^{-2\beta}\right\}\\
	&\leq C_q \sum\limits_{l=l_c(C^*)+1}^{j+1} \Pf^{\frac{1}{p}}\left(\jhat=l\right)\left\{\frac{k_l^*}{n^2}+I^2(k_j^*,k_l^*)+\frac{k_j^*}{n^2}\right\}\\
	&\leq C_q C^* \log{n} R(k_j^*)\sum\limits_{l=l_c(C^*)+1}^{j+1} \Pf^{\frac{1}{p}}\left(\jhat=l\right)\leq C_q C^* \log{n}^{1+\epsilon} R(k_j^*).
	\ee
	Above, the fourth last inequality follows from condition 1 of the theorem, the third last inequality follows condition 2 of the theorem and using the fact that $\beta_{j+1}< \beta \leq \beta_j$ (which implies that $(k_j^*)^{-2\beta} \leq (k_j^*)^{-2\beta_j}\left(\frac{n^2}{\log{n}}\right)^{\frac{2\beta_j-2\beta_{j+1}}{1+4\beta_j}}\leq C\frac{k_j^*}{n^2}n^{\frac{C}{\log{n}}}\leq C'\frac{k_j^*}{n^2} $). The second last inequality follows from our choice of $l_c(C^*)$. The last inequality follows since $\sum\limits_{l=l_c(C^*)+1}^{j+1} \Pf\left(\jhat=l\right)\leq 1$, we have that for $p$ sufficiently close to $1$, one also must have that $\sum\limits_{l=l_c(C^*)+1}^{j+1} \Pf^{\frac{1}{p}}\left(\jhat=l\right)\leq \log{n}^{\epsilon}$. Finally noting that $C_q C^* \log{n} R(k_j^*)\lesssim n^{-\frac{8\beta}{1+4\beta}}\log{n}^{\frac{4\beta}{1+4\beta}}$ completes the required control $T_2$.
	
	\subsection*{Control of $T_3$}
	\be
	T_3&=\sum\limits_{l=j+2}^{N-1}\Ef\left[\I(\jhat=l)\left(U_{n,k_{l}^*}-\phi(f)\right)^2\right]\\
	& \leq \sum\limits_{l=j+2}^{N-1}\Pf^{\frac{1}{p}}\left(\jhat=l\right)\Ef^{\frac{1}{q}}\left[\left(U_{n,k_{l}^*}-\phi(f)\right)^{2q}\right]\\
	& \leq \sum\limits_{l=j+2}^{N-1} \left(\frac{\const}{n}\right)^{\frac{1}{p}}\left[\frac{k_l^*}{n^2}+(k_l^*)^{-2\beta}\right]\leq  \log{n}\left(\frac{\const}{n}\right)^{\frac{1}{p}}
	\ee
	where the second last inequality above follows from Lemma \ref{lemma_l_bigger_than_jplusone} and conditions of the theorem, and the last inequality follows from the choice of $N\lesssim \log{n}$. Therefore for $p$ sufficiently close to $1$, we have desired control over $T_1$. 
\end{proof}

\subsection*{Proof of Theorem \ref{theorem_fcube_moment}}
\begin{proof}
The proof borrows ideas from \cite{rltv2015}. However, the computations are much more cumbersome since now we control moments of a third order U-statistics.

We decompose $U_{n,1}$ as,
$$\widehat{\psi }_{k_{1} ,k_{2},k_{3}}=V_{n,1}+R_{n,1},$$
where
$$V_{n,1}=\nthree \sumijs \left\{\ktilde(X_i,X_j,X_s)\I(\xijs \in \unionm \diagthree)\right\}$$
and 
$$R_{n,1}=\widehat{\psi }_{k_{1} ,k_{2},k_{3}}-V_{n,1}.$$
The sets $\chinm, m=1,\ldots,M_n$ are constructed according to \cite{rltv2015}. Therefore, for Haar basis in one dimension, we can take $M_n=n$ and $\chinm=[\frac{m-1}{n},\frac{m}{n})$ for $m=1,\ldots,M_n$. Therefore,
\be
\Ef\left[\left(\widehat{\psi }_{k_{1} ,k_{2},k_{3}}-\Ef(\widehat{\psi }_{k_{1} ,k_{2},k_{3}})\right)^{2q}\right] & \leq C_q \left\{\Ef\left[\left(V_{n,1}-\Ef(V_{n,1})\right)^{2q}\right] +\Ef\left[\left(R_{n,1}-\Ef(R_{n,1})\right)^{2q}\right] \right\}
\ee
Below, we only control the first summand i.e. $\Ef\left[\left(V_{n,1}-\Ef(V_{n,1})\right)^{2q}\right]$ which suffices for Haar basis. However, for non-Haar bases this is not sufficient. We believe, that for non-Haar bases one can show that the second summand has sub-optimal rate. 
\paragraph*{\textbf{Control of $\mathbf{\Ef\left[\left(V_{n,1}-\Ef(V_{n,1})\right)^{2q}\right]}$}}\hspace*{\fill} \\
Write $V_{n,1}=\sum_{m=1}^{M_n}V_{n,m}$ with 
$$V_{n,m}=\nthree \sumijs \ktilde(X_i,X_j,X_s)\I(\xijs \in \diagthree).$$ 
Following \cite{rltv2015}, we define the following membership based quantities. Let, 
$$I_{n,r}:=m \ \text{if}\ X_r \in \chinm, \ r=1,\ldots,n, \ m=1,\ldots,M_n,$$
$$N_{n,m}:=\# \left(r: I_{n,r}=m\right).$$
Then, $(\Nm,1\leq m \leq M_n) \sim \mathrm{Multinomial}(p_{nm},1\leq m \leq M_n),$
where
$$\pnm = \int_{\chinm} f(x)dx.$$
Note that $\pnm \in [\frac{\fmin}{n},\frac{\fmax}{n}]$ where $\fmin=\min f(x)$.
Given the vector $\In=(I_{n,1},\ldots,I_{n,M_n})$, the observations $X_1,\ldots,X_n$ are independent with distribution of $X_r|I_{n,r}=m$ being $\frac{I_{\chinm}dF}{\pnm}$. Now,
\be
\ & \Ef\left[\left(V_{n,1}-\Ef(V_{n,1})\right)^{2q}\right]\\&\leq C_q \left[\begin{array}{c} \Ef \left[\left(V_{n,1}-\Ef(V_{n,1}|\In)\right)^{2q}\right]\\+\Ef\left[\left(\Ef(V_{n,1}|\In)-\Ef(V_{n,1})\right)^{2q}\right]\end{array}\right]
\ee

\subparagraph*{\textbf{Control of $\mathbf{\Ef\left[\left(\Ef(V_{n,1}|\In)-\Ef(V_{n,1})\right)^{2q}\right]}$}}\hspace*{\fill} \\
Now, note that, 
$$\Ef(V_{n,1}|\In)-\Ef(V_{n,1})=\sum_{m=1}^{M_n}\left(\frac{\Nm(\Nm-1)(\Nm-2)}{n(n-1)(n-2)\pnm^3}-1\right)\alpha_{n,m},$$
where $\alpha_{n,m}=\int \int\ktilde(x_1,x_2,x_3)\I((x_1,x_2,x_3)\in \diagthree)d(F(x_1))d(F(x_2))d(F(x_3))$. The next lemma provides control of the terms $\alpha_{n,m}$.
\begin{lemma}\label{lemma_control_of_alphanm}
For compactly supported wavelet bases one has 
$$\max_{m}|\alpha_{n,m}|\leq \frac{\constf}{M_n}.$$
\end{lemma}
Therefore, we have $\sum_{m=1}^{M_n}|\alpha_{n,m}|\leq \const$. Now, since $\Nm$'s are negatively associated (see definition in \cite{joag1983negative}), we have by Theorem 2 of \cite{shao2000comparison} followed by Marcinkiewicz – Zygmund inequality that,
\be
\ & \Ef\left[\left(\Ef(V_{n,1}|\In)-\Ef(V_{n,1})\right)^{2q}\right] \\& =\Ef\left[\left(\frac{1}{M_n}\sum_{m=1}^{M_n}\left\{\frac{\Nm(\Nm-1)(\Nm-2)}{n(n-1)(n-2)\pnm^3}-1\right\}\alpha_{n,m}\right)^{2q}\right]M_n^{2q}\\
&\leq C_q M_n^{2q} M_n^{-\frac{2q}{2}}\frac{1}{M_n}\sum_{m=1}^{M_n}\Ef\left[\left(\frac{N_m(N_m-1)(\Nm-2)}{n(n-1)(n-2)\pnm^3}-1\right)^{2q}\right]\alpha_{n,m}^{2q}\\
&\leq C_q^{*} M_n^{q-1}\sum_{m=1}^{M_n}|\alpha_{n,m}|^{2q}  \label{eqn:multinomial_moment_3rdorder}\\
&\leq C_q^{*} M_n^{q-1}\{\max(\alpha_{n,m})\}^{2q-1}\sum_{m=1}^{M_n}|\alpha_{n,m}|\leq \const^{q} M_n^{q-1}\{\max(\alpha_{n,m})\}^{2q-1}
\ee
In the above display, equation \ref{eqn:multinomial_moment_3rdorder} follows from Lemma 5.6 of \cite{rltv2015}. Now, by Lemma \ref{lemma_control_of_alphanm},  $\{\max(\alpha_{n,m})\}^{2q-1}\leq (\constf)^{2q-1}\left(\frac{1}{M_n}\right)^{2q-1}$. Therefore, $\Ef\left[\left(\Ef(V_{n,1}|\In)-\Ef(V_{n,1})\right)^{2q}\right]\leq (\const\fmax)^{2q-1}\left(\frac{1}{M_n}\right)^{q}\leq (\const\fmax)^{2q-1}\left(\frac{k_3}{n^2}\right)^{q} $ for our choice of $M_n\asymp n$.

\subparagraph*{\textbf{Control of $\mathbf{\Ef\left[\left(V_{n,1}-\Ef(V_{n,1}|\In)\right)^{2q}\right]}$}}\hspace*{\fill} \\
Suppose, $V_{n,m}=\vmone+\vmtwo+\vmthree\vmcond$ denote the Hoeffding decomposition of $V_{n,m}$ w.r.t the conditional distribution given $\In$. Therefore,
\be 
\Ef\left[\left(V_{n,1}-\Ef(V_{n,1}|\In)\right)^{2q}\right]&=
\Ef\left[\left(\summ \left\{\vmone+\vmtwo+\vmthree \right\}\right)^{2q}\right]\\
%&=\Ef\Ef\left[\left(\summ \left\{\vmone+\vmtwo \right\}\right)^{2q}|\In\right]\\
\ee
Note that $\Ef\left(\vmone|\In\right)=0$, $\Ef\left(\vmtwo|\In\right)=0$, $\Ef\left(\vmthree|\In\right)=0$ and they are independent over $m$ conditional on $\In$. Hence by a conditional version of Rosenthal's Inequality (noting that the constants of the inequality does not depend on the underlying distribution), we have 
\be
\ & \Ef\left[\left(\summ \left\{\vmone+\vmtwo+\vmthree \right\}\right)^{2q}\right] \\&\leq C_q \Ef\left[\summ \Ef\left(|\vmone+\vmtwo+\vmthree|^{2q}|\In\right)+\left\{\summ \Ef\left(|\vmone+\vmtwo+\vmthree|^{2} \In \right)\right\}^{\frac{2q}{2}} \right] \\
&= C_q \left[\summ \Ef\left(|\vmone+\vmtwo+\vmthree|^{2q}\right)+\Ef\left\{\summ \Ef\left(|\vmone+\vmtwo+\vmthree|^{2}| \In \right)\right\}^{\frac{2q}{2}} \right] \\
\label{eqn:3rorder_qth_moment_main}
\ee
Below we control $$\summ \Ef\left(|\vmone+\vmtwo+\vmthree|^{2q}\right) \ \text{and} \ \Ef\left\{\summ \Ef\left(|\vmone+\vmtwo+\vmthree|^{2}| \In \right)\right\}^{\frac{2q}{2}}$$ separately, at the required rate.

\paragraph*{\textbf{Control of $\mathbf{\summ \Ef\left(|\vmone+\vmtwo+\vmthree|^{2q}\right)}$}}\hspace*{\fill} \\
We now provide a general control over terms like $\Ef\left(|\vmone+\vmtwo+\vmthree|^{2q}\right)$. To this end, note that, it is enough to control $\Ef\left((\vmone)^{2q}\right)$, $\Ef\left((\vmtwo)^{2q}\right)$, and $\Ef\left((\vmthree)^{2q}\right)$ separately. We have dropped the absolute value sign assuming without loss of generality that $q$ is a sufficiently large integer.
\subparagraph*{\textbf{Control of $\mathbf{\Ef\left((\vmone)^{2q}\right)}$}}\hspace*{\fill} \\

Note that, by Hoeffding decomposition of U-statistics with asymmetric kernel, we have
\be
\vmone&=\frac{\Nm(\Nm-1)(\Nm-2)}{n(n-1)(n-2)}\times\\
 & \frac{1}{\Nm}\sum_{i=1}^{\Nm}\left\{
\begin{array}{c}
	\Ejs \left[\ktilde\left(X_i,X_j,X_s\right)\I(\xijs \in \diagthree)|\In\right]\\
	+\Ejs \left[\ktilde\left(X_i,X_s,X_j\right)\I(\xijs \in \diagthree)|\In\right]\\
	+\Ejs \left[\ktilde\left(X_j,X_i,X_s\right)\I(\xijs \in \diagthree)|\In\right]\\
	+\Ejs \left[\ktilde\left(X_s,X_i,X_j\right)\I(\xijs \in \diagthree)|\In\right]\\
	+\Ejs \left[\ktilde\left(X_j,X_s,X_i\right)\I(\xijs \in \diagthree)|\In\right]\\
	+\Ejs \left[\ktilde\left(X_s,X_j,X_i\right)\I(\xijs \in \diagthree)|\In\right]\\
	-6\Ef \left[\ktilde\left(X_i,X_j,X_s\right)\I(\xijs \in \diagthree)|\In\right]
\end{array}
	\right\}
\ee
By an application of conditional version of Marcinkiewicz – Zygmund Inequality (noting that the constants of the inequality does not depend on the underlying distribution) we have that
\be
\ & \Ef\left((\vmone)^{2q}|\In\right)\\ &\leq C_q \frac{\Nm^{2q}(\Nm-1)^{2q}(\Nm-2)^{2q}}{n^{2q}(n-1)^{2q}(n-2)^{2q}}\frac{\Nm^{-\frac{2q}{2}}}{\Nm}\times \\
&\sum\limits_{i=1}^{\Nm}\Ef\left(\left\{
\begin{array}{c}
	\Ejs \left[\ktilde\left(X_i,X_j,X_s\right)\I(\xijs \in \diagthree)|\In\right]\\
	+\Ejs \left[\ktilde\left(X_i,X_s,X_j\right)\I(\xijs \in \diagthree)|\In\right]\\
	+\Ejs \left[\ktilde\left(X_j,X_i,X_s\right)\I(\xijs \in \diagthree)|\In\right]\\
	+\Ejs \left[\ktilde\left(X_s,X_i,X_j\right)\I(\xijs \in \diagthree)|\In\right]\\
	+\Ejs \left[\ktilde\left(X_j,X_s,X_i\right)\I(\xijs \in \diagthree)|\In\right]\\
	+\Ejs \left[\ktilde\left(X_s,X_j,X_i\right)\I(\xijs \in \diagthree)|\In\right]\\
	-6\Ef \left[\ktilde\left(X_i,X_j,X_s\right)\I(\xijs \in \diagthree)|\In\right]
\end{array}
\right\}^{2q}\mid\In\right)\\
& \leq  C_q \frac{\Nm^{2q}(\Nm-1)^{2q}(\Nm-2)^{2q}}{n^{2q}(n-1)^{2q}(n-2)^{2q}}\frac{\Nm^{-\frac{2q}{2}}}{\Nm}\times \\
&\sum\limits_{i=1}^{\Nm}\left(
\begin{array}{c}
	\Ef\left(\left\{\Ejs \left[\ktilde\left(X_i,X_j,X_s\right)\I(\xijs \in \diagthree)|\In\right]\right\}^{2q}|\In\right)\\
	+\Ef\left(\left\{\Ejs \left[\ktilde\left(X_i,X_s,X_j\right)\I(\xijs \in \diagthree)|\In\right]\right\}^{2q}|\In\right)\\
	+\Ef\left(\left\{\Ejs \left[\ktilde\left(X_j,X_i,X_s\right)\I(\xijs \in \diagthree)|\In\right]\right\}^{2q}|\In\right)\\
	+\Ef\left(\left\{\Ejs \left[\ktilde\left(X_s,X_i,X_j\right)\I(\xijs \in \diagthree)|\In\right]\right\}^{2q}|\In\right)\\
	+\Ef\left(\left\{\Ejs \left[\ktilde\left(X_j,X_s,X_i\right)\I(\xijs \in \diagthree)|\In\right]\right\}^{2q}|\In\right)\\
	+\Ef\left(\left\{\Ejs \left[\ktilde\left(X_s,X_j,X_i\right)\I(\xijs \in \diagthree)|\In\right]\right\}^{2q}|\In\right)\\
\end{array}
\right)\\
\label{eqn:vmone_threedim_hoeff_first_moment}
\ee
Above, the last inequality follows from conditional Jensen's Inequality. A typical term in the above summand looks like 
\be
 \frac{1}{\pnm^{4q+1}}\int_{\chinm}\left[\int_{\chinm}\int_{\chinm} \int \sum_{l_1}\sum_{l_2}\sum_{l_3}\left(\begin{array}{c} \psi_{l_1}^{k_1}(x)\psi_{l_1}^{k_1}(x_1)\psi_{l_2}^{k_2}(x)\psi_{l_2}^{k_2}(x_2) \\
 	\psi_{l_3}^{k_3}(x)\psi_{l_3}^{k_3}(x_3)f(x_2)f(x_3)\end{array}\right)dx  dx_2 dx_3\right]^{2q}f(x_1)dx_1 \\
 \label{eqn:vmone_threedim_hoeff_first_moment_typical}
\ee
for some tuple $k_1,k_2,k_3$. To evaluate each of the above, we consider three different ordering of $k_1,k_2,k_3$. Let us consider the the integrals in the square bracket first. Since here $x_1$ is fixed, we look for the subinterval of resolution $k_1$ where $x_1$ lies. In each of the cases we look at the specific subinterval containing $x_1$ for every fixed $x_1$. Further this subinterval needs to intersect $\chinm$. Say, we index that by $k_1$. Call this subinterval $S_{k_1}(x_1)$. Taking the integral inside the sum over location parameters, we have that each summand is bounded by $\constf \frac{k_1 k_2 k_3}{(\max(k_1,k_2,k_3))^3}$. For each fixed $x_1$, the number of summands contributing will be bounded by $\const\frac{\max(k_1,k_2,k_3)}{k_1}$. Therefore, whichever interval $x$ must lie it should intersect that containing $x_i$. Therefore, the term in the square bracket is bounded by $\const$ for each fixed $x_1$. Raising to the power of $2q$ and integrating with respect to $f$ over $\chinm$ yields the outer integral to be bounded by $\constf\pnm$. Therefore, \eqref{eqn:vmone_threedim_hoeff_first_moment_typical} is always bounded by $\frac{\constf}{\pnm^{4q}}$. Therefore, by \eqref{eqn:vmone_threedim_hoeff_first_moment} we have that almost surely,
\be
\Ef\left((\vmone)^{2q}|\In\right)& \leq C_q\constf \frac{\Nm^{2q}(\Nm-1)^{2q}(\Nm-2)^{2q}}{n^{2q}(n-1)^{2q}(n-2)^{2q}}\frac{\Nm^{-\frac{2q}{2}}}{\Nm}\times \frac{\Nm}{\pnm^{2q}}\\ \label{eqn:vmone_third_order_pre_final}
\ee
Now, noting that $\pnm \geq \frac{\fmin}{n}$ and using Lemma 5.6 of \cite{rltv2015}, we have that 
\be
\Ef\Ef\left((\vmone)^{2q}|\In\right)& \leq \frac{C_q(\psimax,\fmax,\fmin)}{n^{2q}}, \label{eqn:vmone_third_order_final_bound}
\ee
which is at the desired rate of control.
\subparagraph*{\textbf{Control of $\mathbf{\Ef\left((\vmtwo)^{2q}\right)}$}}\hspace*{\fill} \\
For the second term of the U-statistics it is enough to control for the expectation of 
\be
I^{(m)}_{21} &= \frac{\left(\Nm(\Nm-1)(\Nm-1)\right)^{2q}}{n^{2q}(n-1)^{2q}(n-2)^{2q}}\\
&\times \left[\frac{1}{\Nm}\left(\int_{\chinm}\int_{\chinm}(\Ei(\ktilde\xijs\I(\xijs \in \diagthree)|\In))^2\frac{f(x_j)f(x_s)}{\pnm^2}dx_j dx_s\right)^{\frac{2q}{2}}\right]
\ee
and 

\be
I^{(m)}_{22} &= \frac{\left(\Nm(\Nm-1)(\Nm-1)\right)^{2q}}{n^{2q}(n-1)^{2q}(n-2)^{2q}}\\
&\times \left[\frac{1}{\Nm}\left(\int_{\chinm}\int_{\chinm}(\Ei(\ktilde\xijs\I(\xijs \in \diagthree)|\In))^{2q}\frac{f(x_j)f(x_s)}{\pnm^2}dx_j dx_s\right)\right]
\ee

\subparagraph*{\textbf{Control of $I^{(m)}_{21}$}}\hspace*{\fill} \\
First note that one can argue by simply looking at orders of truncation that, 
\be
\ & \Ei(\ktilde\xijs\I(\xijs \in \diagthree)|\In)\\
&\asymp\frac{1}{\pnm}\left[\begin{array}{c}
  \sum_{l_1}\sum_{l_2}\sum_{l_3}\int_{\chinm} \int \psi_{l_1}^{k_1}(x)\psi_{l_1}^{k_1}(x_i)\psi_{l_2}^{k_3}(x)\psi_{l_2}^{k_3}(x_j)\psi_{l_3}^{k_3}(x)\psi_{l_3}^{k_3}(x_s)f(x_1)dx dx_1 \\
+  \sum_{l_1}\sum_{l_2}\sum_{l_3}\int_{\chinm} \int \psi_{l_1}^{k_3}(x)\psi_{l_1}^{k_3}(x_i)\psi_{l_2}^{k_1}(x)\psi_{l_2}^{k_1}(x_j)\psi_{l_3}^{k_3}(x)\psi_{l_3}^{k_3}(x_s)f(x_1)dx dx_1
\end{array}\right]
\ee
Let us control the first term in the square bracket first. For each fixed value of $(x_j,x_s)$, find the boxes of length $\frac{1}{k_3}$ which contains $x_j$ and $x_s$. If these boxes are disjoint then they contribute to the summand of the first term. Therefore, only way of getting a contribution in the first term is when $x_j$ and $x_s$ are in the same box of resolution $k_3$. Therefore, $x$ must also belong to the same box of resolution $k_3$ to contribute to the summand. Therefore, the first term inside the square bracket is always bounded by $ \constf\frac{k_1k_3^2}{k_3 M_n}$ for every choice $(x_j,x_s)$ lying in the same $k_3$ resolution box and $0$ otherwise. Therefore, the integral of the square of the first term over $(x_j,x_s)\in \diag$ is bounded by $\constf \left(\frac{k_1k_3^2}{k_3 M_n}\right)^2 \frac{1}{k_3^2}\frac{k_3}{M_n}=\constf \frac{k_1^2 k_3}{M_n^3}$. Now, let us look at the second term of the square bracket. Once again fix a  $(x_j,x_s)\in \diag$ and find the $k_3$ resolution box that $x_s$ belongs to. Now note that $x_i,x$ must belong to the same $k_3$ resolution box and hence the number of summands contributing to the sum in second term of the square bracket is again bounded by $\const$. Therefore, the second term inside the square bracket is always bounded by $ \constf\frac{k_1k_3^2}{k_3^2}$ for every choice $(x_j,x_s)$ with $x_s$ lying in the some $k_3$ resolution box and $0$ otherwise. Therefore, the integral of the square of the second term over $(x_j,x_s)\in \diag$ is bounded by $\constf \left(\frac{k_1k_3^2}{k_3^2}\right)^2 \frac{k_3}{M_n}\frac{1}{k_3}\frac{1}{M_n}=\constf \frac{k_1^2 }{M_n^2}$. Since, by our choice, $k_3 \gg M_n$, the first term of the square bracket dominates after squaring and integrating over over $(x_j,x_s)\in \diag$. Taking into account the division by $\pnm$'s, the final contribution to $I_{21}^{(m)}$ is bounded $\left(\constf \frac{k_1^2 k_3}{M_n^3}\right)^q\frac{1}{\pnm^{4q}}$. Therefore, 
\be
I_{21}^{(m)} & \leq C(\psimax,\fmax,\fmin)\left(\Nm(\Nm-1)(\Nm-1)\right)^{2q}\left(\frac{k_1^2 k_3}{M_n^3}\right)^q\frac{1}{n^{2q}}\\
&\leq C(\psimax,\fmax,\fmin)\left(\Nm(\Nm-1)(\Nm-1)\right)^{2q} \left(\frac{k_3}{n^3}\right)^{q}.\\
\label{eqn:Imtwoone_final}
\ee
Taking expectations by using Lemma 5.6 of \cite{rltv2015} yields the desired control.
\subparagraph*{\textbf{Control of $I^{(m)}_{22}$}}\hspace*{\fill} \\
The calculation technique of this term is similar to that of  $I_{21}^{(m)}$. The only difference is that instead of taking square of the square bracket term we take the $2q^{\mathrm{th}}$ power. Hence, the integral over $(x_j,x_s)\in \diag$ is bounded by $$\constf \left[\left(\frac{k_1k_3^2}{k_3 M_n}\right)^{2q} \frac{1}{k_3^2}\frac{k_3}{M_n}+\left(\frac{k_1k_3^2}{k_3^2}\right)^{2q} \frac{k_3}{M_n}\frac{1}{k_3}\frac{1}{M_n}\right].$$
Taking into account the division by $\pnm$'s, the final contribution to $I_{22}^{(m)}$ is bounded $$\constf \left[\left(\frac{k_1k_3^2}{k_3 M_n}\right)^{2q} \frac{1}{k_3^2}\frac{k_3}{M_n}+\left(\frac{k_1k_3^2}{k_3^2}\right)^{2q} \frac{k_3}{M_n}\frac{1}{k_3}\frac{1}{M_n}\right]\frac{1}{\pnm^{2q+2}}.$$ 
Therefore, 
\be
I_{21}^{(m)} & \leq C(\psimax,\fmax,\fmin)\left(\Nm(\Nm-1)(\Nm-2)\right)^{2q}\left[\left(\frac{k_1k_3^2}{k_3 M_n}\right)^{2q} \frac{1}{k_3^2}\frac{k_3}{M_n}+\left(\frac{k_1k_3^2}{k_3^2}\right)^{2q} \frac{k_3}{M_n}\frac{1}{k_3}\frac{1}{M_n}\right]\frac{1}{n^{4q-2}}\\
&\leq C(\psimax,\fmax,\fmin)\left(\Nm(\Nm-1)(\Nm-2)\right)^{2q}\frac{1}{M_n} \left(\frac{k_3}{n^2}\right)^{2q-1}.\\
\label{eqn:Imtwotwo_final}
\ee
Taking expectations by using Lemma 5.6 of \cite{rltv2015} yields the desired control.

\subparagraph*{\textbf{Combining Controls of $I^{(m)}_{21}$ and $I^{(m)}_{22}$ for control of $\mathbf{\Ef\left((\vmtwo)^{2q}\right)}$}}\hspace*{\fill} \\
We get 
\be
\mathbf{\Ef\left((\vmtwo)^{2q}\right)}& \leq C(\psimax,\fmax,\fmin)\left[\frac{1}{M_n} \left(\frac{k_3}{n^2}\right)^{2q-1}+\left(\frac{k_3}{n^3}\right)^{q}\right]. \label{eqn:vmtwo_third_order_final_bound}
\ee
\subparagraph*{\textbf{Control of $\mathbf{\Ef\left((\vmthree)^{2q}\right)}$}}\hspace*{\fill} \\
For this part we will need a moment bound for third order U-statistics. We do that in the following Lemma.

\begin{lemma}\label{lemma_moment_third_order_U_stat}
	For any $q>2,$ there exists a constant C$_{q}$ such that for any i.i.d
	random variables $X_{1},X_{2},...,X_{m}$ and degenerate symmetric kernel $K.$%
	\begin{eqnarray*}
		&&E\left\vert \frac{1}{m\left( m-1\right) \left( m-2\right) }%
		\sum\limits_{i_{1}\neq i_{2}\neq i_{3}}K\left(
		X_{i_{1}},X_{i_{2}},X_{i_{3}}\right) \right\vert ^{q} \\
		&\leq &C_{q}m^{-q}E\left\vert K\left( X_{i_{1}},X_{i_{2}},X_{i_{3}}\right)
		\right\vert ^{q}
	\end{eqnarray*}
\end{lemma}

 We will now apply Lemma \ref{lemma_moment_third_order_U_stat} to control moments of the degenerate part of the kernel $K(X_i,X_j,X_s)=\int K_{k_1}(x,X_i)K_{k_3}(x,X_j)K_{k_3}(x,X_s)I(\xijs \in \diagthree)dx$. Standard arguments show that this is the term with the highest contribution among all the different terms of $\vmthree$. Denoting the degenerate part of this kernel by $K_d^{(m)}$ we can show by Lemma \ref{lemma_moment_third_order_U_stat} and a standard contraction of norm under conditional expectation argument that 
 \be
 \ & \Ef\left((\vmthree)^{2q}\mid \In\right) \\
 & \leq \frac{\constf^q\left(\Nm(\Nm-1)(\Nm-2)\right)^{2q}}{n^{2q}(n-1)^{2q}(n-2)^{2q}}\\
 & \times \frac{E\left\vert K_d^{(m)}(X_1,X_2,X_3)\right\vert ^{2q}}{\Nm^{2q}}\\
 & \leq \frac{\constf^q\left(\Nm(\Nm-1)(\Nm-2)\right)^{2q}}{n^{2q}(n-1)^{2q}(n-2)^{2q}}\\
 & \times \frac{E\left\vert \int K_{k_1}(x,X_1)K_{k_3}(x,X_2)K_{k_3}(x,X_3)I(\xijs \in \diagthree)dx
 	\right\vert ^{2q}}{\Nm^{2q}}\\
 &=\frac{\constf^q\left((\Nm-1)(\Nm-2)\right)^{2q}}{n^{2q}(n-1)^{2q}(n-2)^{2q}}\\
 & \times E\left\vert \int K_{k_1}(x,X_1)K_{k_3}(x,X_2)K_{k_3}(x,X_3)I(\xijs \in \diagthree)dx
 	\right\vert ^{2q} \\ \label{eqn:vmthree_main}
\ee
Therefore it is enough to control $$E\left\vert \int K_{k_1}(x,X_1)K_{k_3}(x,X_2)K_{k_3}(x,X_3)I(\xijs \in \diagthree)dx
\right\vert ^{2q}$$ at a level so that the expectation of the right hand side of \eqref{eqn:vmthree_main} is controlled at the desired level of $\left(\frac{k_3}{n^2}\right)^q$. We do this below.

\paragraph*{\textbf{Control of $\mathbf{E\left\vert \int K_{k_1}(x,X_1)K_{k_3}(x,X_2)K_{k_3}(x,X_3)I(\xijs \in \diagthree)dx
			\right\vert ^{2q}}$}}\hspace*{\fill} \\

\be
\ & {E\left\vert \int K_{k_1}(x,X_1)K_{k_3}(x,X_2)K_{k_3}(x,X_3)I(\xijs \in \diagthree)dx
	\right\vert ^{2q}}\\
&\leq  \int_{\chinm}\int_{\chinm}\int_{\chinm}\left[\sum_{l_1}\sum_{l_2}\sum_{l_3}\int
	\left\vert \begin{array}{c}
	\psi_{l_1}^{k_1}(x)\psi_{l_1}^{k_1}(x_1)\psi_{l_2}^{k_3}(x)\\
	\psi_{l_2}^{k_3}(x_2)\psi_{l_3}^{k_3}(x)\psi_{l_3}^{k_3}(x_3)
	\end{array}\right\vert dx	
\right]^{2q}\frac{f(x_1)f(x_2)f(x_3)}{\pnm^3} dx_1dx_2 dx_3
\ee
Now, for each fixed value of $(x_1,x_2,x_3)$, find the boxes of length $\frac{1}{k_3}$ which contains $x_2$ and $x_3$. Therefore, only way of getting a contribution is when $x_2$ and $x_3$ are in the same box of resolution $k_3$. Therefore, $x$ must also belong to the same box of resolution $k_3$ to contribute to the summand. Therefore, the term inside the square bracket is always bounded by $ \left(\const\frac{k_1k_3^2}{k_3}\right)^{2q}$ for every choice $(x_2,x_3)$ lying in the same $k_3$ resolution box and $0$ otherwise. Therefore, the final integral with respect to $(x_1,x_2,x_3)$ is bounded by $\constf^q\frac{\left(\frac{k_1k_3^2}{k_3}\right)^{2q}}{\pnm^3}\frac{1}{M_n}\frac{1}{k_3^2}\frac{k_3}{n}=C_q(\psimax,\fmax,\fmin){k_3^{2q-1}n^{2q+1}}$.

\paragraph*{\textbf{Combining with \eqref{eqn:vmthree_main}}}
Therefore, from \eqref{eqn:vmthree_main}, we have that almost surely, 
\be
\ & \Ef\left((\vmthree)^{2q}\mid \In\right) \\
& \leq \frac{\constf^q\left((\Nm-1)(\Nm-2)\right)^{2q}}{n^{2q}(n-1)^{2q}(n-2)^{2q}}\\
& \times E\left\vert \int K_{k_1}(x,X_1)K_{k_3}(x,X_2)K_{k_3}(x,X_3)I(\xijs \in \diagthree)dx
\right\vert ^{2q} \\
& \leq C_q(\psimax,\fmax,\fmin) \frac{\left((\Nm-1)(\Nm-2)\right)^{2q}}{n^{2q}(n-1)^{2q}(n-2)^{2q}}{k_3^{2q-1}n^{2q+1}}\\
&\leq C_q(\psimax,\fmax,\fmin)\left((\Nm-1)(\Nm-2)\right)^{2q} \frac{k_3^{2q-1}}{n^{4q-1}}\\
\label{eqn:vmthree_third_order_pre_final_bound}
\ee
Therefore, by Lemma 5.6 of \cite{rltv2015} we have that,
\be
\Ef\left((\vmthree)^{2q}\right) & \leq C_q(\psimax,\fmax,\fmin) \frac{k_3^{2q-1}}{n^{4q-1}}=C_q(\psimax,\fmax,\fmin)\left(\frac{k_3}{n^2}\right)^{2q-1}\frac{1}{n} \\ \label{eqn:vmthree_third_order_final_bound}
\ee

Combining \eqref{eqn:vmone_third_order_final_bound}, \eqref{eqn:vmtwo_third_order_final_bound} and \eqref{eqn:vmthree_third_order_final_bound} and summing over $M_n=O(n)$ partitions we get the desired bound that
\be
\summ \Ef\left(|\vmone+\vmtwo+\vmthree|^{2q}\right) & \leq  C_q(\psimax,\fmax,\fmin) \left(\frac{k_3}{n^2}\right)^{2q-1}.
\ee

\paragraph*{\textbf{Control of $\mathbf{\Ef\left\{\summ \Ef\left(|\vmone+\vmtwo+\vmthree|^{2} |\In \right)\right\}^{\frac{2q}{2}}}$}}\hspace*{\fill} \\
Unlike the previous case, here we will will employ a conditional version of Gine-Latala inequality for moments of U-statistics along with the observation that the constants of the inequality does not depend on the underlying distribution. The only difference will be that we will have to do an extra step of taking the expectation of the moment of sum of functions of multinomial coordinate statistics. To deal with this we again use the Marcinkiewicz – Zygmund Inequality following an use of Theorem 2 of \cite{shao2000comparison}. First, note that it is enough to control $\Ef\left\{\summ \Ef\left((\vmone)^{2} |\In \right)\right\}^{\frac{2q}{2}}$,\\ $\Ef\left\{\summ \Ef\left((\vmtwo)^{2} |\In \right)\right\}^{\frac{2q}{2}}$, $\Ef\left\{\summ \Ef\left((\vmthree)^{2} |\In \right)\right\}^{\frac{2q}{2}}$ separately at the desired rate.
\subparagraph*{\textbf{Control of $\mathbf{\Ef\left\{\summ \Ef\left((\vmone)^{2} |\In \right)\right\}^{\frac{2q}{2}}}$}}\hspace*{\fill} \\
Note that by \eqref{eqn:vmone_third_order_pre_final},
\be
\Ef\left((\vmone)^{2q}|\In\right)& \leq \constf^q \frac{\Nm^{q}(\Nm-1)^{2q}(\Nm-2)^{2q}}{\pnm^{2q}n^{2q}(n-1)^{2q}(n-2)^{2q}}\\
& \leq C_q(\psimax,\fmax, \fmin) \frac{\Nm^{q}(\Nm-1)^{2q}(\Nm-2)^{2q}}{(n-1)^{2q}(n-2)^{2q}}
\ee
Therefore,
\be
\Ef\left((\vmone)^{2}|\In\right)&\leq  C_2(\psimax,\fmax, \fmin) \frac{\Nm(\Nm-1)^{2}(\Nm-2)^{2}}{(n-1)^{2}(n-2)^{2}}
\ee

Therefore, by Theorem 2 of \cite{shao2000comparison}, 
\be
\Ef\left\{\summ \Ef\left((\vmone)^{2} |\In \right)\right\}^{\frac{2q}{2}} & \leq C_2^*\Ef\left\{\summ\frac{\Nm'(\Nm'-1)^{2}(\Nm'-2)^{2}}{n^{4}}\right\}^{\frac{2q}{2}}
\ee
where $\Nm'$, $m=1,\ldots,M_n$ are independent with the same marginals as $\Nm$, $m=1,\ldots,M_n$. Now, 
\be
\ & \Ef\left\{\summ\frac{\Nm'(\Nm'-1)^{2}(\Nm'-2)^{2}}{n^{4}}\right\}^{\frac{2q}{2}} \\& \leq C_q \left[\begin{array}{c}\Ef\left\{\summ\left(\frac{\Nm'(\Nm'-1)^{2}(\Nm'-2)^{2}}{n^{4}}-\Ef\left(\frac{\Nm'(\Nm'-1)^{2}(\Nm'-2)^{2}}{n^{4}}\right)\right)\right\}^{q}\\+\left\{\summ\Ef\left(\frac{\Nm'(\Nm'-1)^{2}(\Nm'-2)^{2}}{n^{4}}\right)\right\}^{q}\end{array}\right]\\
\label{eqn:vmone_third_order_conditional_main}
\ee
The first term in the above summand can be controlled by Marcinkiewicz – Zygmund Inequality as follows,
\be
\ &  \Ef\left\{\summ\left(\frac{\Nm'(\Nm'-1)^{2}(\Nm'-2)^{2}}{n^{4}}-\Ef\left(\frac{\Nm'(\Nm'-1)^{2}(\Nm'-2)^{2}}{n^{4}}\right)\right)\right\}^{q}\\
& \leq C_q M_n^{q}M_n^{-\frac{q}{2}}\frac{1}{M_n}\summ 
\Ef\left(\frac{\Nm'(\Nm'-1)^{2}(\Nm'-2)^{2}}{n^{4}}-\Ef\left(\frac{\Nm'(\Nm'-1)^{2}(\Nm'-2)^{2}}{n^{4}}\right)\right)^q\\
& \leq C_q' M_n^{\frac{q}{2}-1}\summ \Ef\left\{\left(\frac{\Nm'(\Nm'-1)^{2}(\Nm'-2)^{2}}{n^{4}}\right)\right\}^q \label{eqn:vmone_third_order_conditional_part1}
\ee
Therefore, it is enough to control $\Ef\left\{\left(\frac{\Nm'(\Nm'-1)^{2}(\Nm'-2)^{2}}{n^{4}}\right)\right\}^q$ for this part of the proof. But, by Lemma 5.6 of \cite{rltv2015}, we have that there exists a constant $C_q(\fmax,\psi_0)$ such that
\be
\Ef\left\{\left(\frac{\Nm'(\Nm'-1)^{2}(\Nm'-2)^{2}}{n^{4}}\right)\right\}^q \leq \frac{C_q(\fmax,\psi_0)}{n^{4q}}. \label{eqn:vmone_third_order_conditional_part1_final}
\ee
By a similar argument,
\be
\Ef\left(\frac{\Nm'(\Nm'-1)^{2}(\Nm'-2)^{2}}{n^{4}}\right) & \leq \frac{C_q(\fmax,\psi_0)}{n^4}\label{eqn:vmone_third_order_conditional_part2_final}
\ee
Combining \eqref{eqn:vmone_third_order_conditional_main} with \eqref{eqn:vmone_third_order_conditional_part1}, \eqref{eqn:vmone_third_order_conditional_part1_final} and \eqref{eqn:vmone_third_order_conditional_part2_final}, we have, 
\be
\ & \Ef\left\{\summ\frac{\Nm'(\Nm'-1)^{2}(\Nm'-2)^{2}}{n^{4}}\right\}^{\frac{2q}{2}} \\
& \leq C_q\left[M_n^{\frac{q}{2}-1}M_n \frac{C_q(\fmax,\psi_0)}{n^{4q}}+\left(M_n \frac{C_q(\fmax,\psi_0)}{n^4}\right)^{q} \right]\leq C_q'(\fmax,\psi_0)\frac{k_3}{n^2}.
\ee
This completes the control of $\Ef\left\{\summ \Ef\left((\vmone)^{2} |\In \right)\right\}^{\frac{2q}{2}}$.

\subparagraph*{\textbf{Control of $\mathbf{\Ef\left\{\summ \Ef\left((\vmtwo)^{2} |\In \right)\right\}^{\frac{2q}{2}}}$}}\hspace*{\fill} \\
Note that by \eqref{eqn:Imtwoone_final} and \eqref{eqn:Imtwotwo_final},
\be
\ & \Ef\left\{\summ \Ef\left((\vmtwo)^{2} |\In \right)\right\}^{\frac{2q}{2}}\\
& \leq C_q(\psimax,\fmax,\fmin)\Ef\left[\begin{array}{c}\left(\frac{k_3}{n^3}\right)\summ \left(\Nm(\Nm-1)(\Nm-2)\right)^{2}\\+\left(\frac{k_3}{n^2}\right)\frac{1}{M_n} \summ\left(\Nm(\Nm-1)(\Nm-2)\right)^{2}\end{array}\right]^q\\
& \leq C_q(\psimax,\fmax,\fmin)\Ef\left[\begin{array}{c}\left(\frac{k_3}{n^3}\right)\summ \left(\Nm'(\Nm'-1)(\Nm'-2)\right)^{2}\\+\left(\frac{k_3}{n^2}\right)\frac{1}{M_n} \summ\left(\Nm'(\Nm'-1)(\Nm'-2)\right)^{2}\end{array}\right]^q\\
\label{eqn:vmtwo_third_order_contribution_square_term}
\ee
where the last inequality in the above display is again by Theorem 2 of \cite{shao2000comparison}, 
 with $\Nm'$, $m=1,\ldots,M_n$ are independent with the same marginals as $\Nm$, $m=1,\ldots,M_n$. Proceeding as the last section, we can apply the Marcinkiewicz – Zygmund Inequality to a centered version of the right hand side of above and obtain that 
 \be
 \ & \Ef\left\{\summ \Ef\left((\vmtwo)^{2} |\In \right)\right\}^{\frac{2q}{2}}\\
  & \leq C_q(\psimax,\fmax,\fmin)\Ef\left[\begin{array}{c}\left(\frac{k_3}{n^3}\right)\summ \left(\Nm'(\Nm'-1)(\Nm'-2)\right)^{2}\\+\left(\frac{k_3}{n^2}\right)\frac{1}{M_n} \summ\left(\Nm'(\Nm'-1)(\Nm'-2)\right)^{2}\end{array}\right]^q\\
  & \leq C_q(\psimax,\fmax,\fmin)\left[\begin{array}{c}\left(\frac{k_3}{n^3}\right)^q\left\{M_n^q M_n^{-\frac{q}{2}-1}+M_n^q\right\} \\+ \left(\frac{k_3}{n^2}\right)^q\left\{ M_n^{-\frac{q}{2}-1}+1\right\}\end{array}\right] \\
  &\leq C_q(\psimax,\fmax,\fmin) \left(\frac{k_3}{n^2}\right)^q
 \ee
\subparagraph*{\textbf{Control of $\mathbf{\Ef\left\{\summ \Ef\left((\vmthree)^{2} |\In \right)\right\}^{\frac{2q}{2}}}$}}\hspace*{\fill} \\
Note that by \eqref{eqn:vmthree_third_order_pre_final_bound},
\be
\Ef\left((\vmthree)^{2}\mid \In\right) & \leq C_2(\psimax,\fmax,\fmin)\left((\Nm-1)(\Nm-2)\right)^{2q} \frac{k_3}{n^3} 
\ee
which is similar to the first term of \eqref{eqn:vmtwo_third_order_contribution_square_term} which we controlled earlier.

\end{proof}

\subsection*{Proof of Theorem \ref{theorem_adptation_cubic}}
\begin{proof}
The proof is a simple application of Theorem \ref{theorem_modified_lepski} with slight modification to account for the fact that the estimator depends on three truncation points $(k_1,k_2,k_3)$ instead of single level of truncation parametrization considered in Theorem \ref{theorem_modified_lepski}. However, as noted in equations \eqref{eqn:bias_third_order}, \ref{eqn:comparative_bias_third_order}, and Theorem \ref{theorem_fcube_moment}, the required problem can be equivalently parametrized by the highest order of truncation i.e. $k_3$. Since, $k_3$ decides a corresponding smoothness index $\beta$ as the solution of $k_3=n^{\frac{1}{1+4\beta}}$, this also decides the other two levels of truncation. As a result, the proof of Theorem \ref{theorem_modified_lepski} goes through in a similar fashion. We do not provide the details here for the sake of brevity.
\end{proof}

\section{Proof of Lemmas}
\subsection*{Proof of Lemma \ref{lemma_wrong_jhat}}
\begin{proof}
\be
\Pf\left(\I(\jhat=1)\right) & \leq \Pf\left[|\ihatk-\Ik|\geq C\sqrt{\log{n}}\frac{\sqrt{\kstar}}{n}\right]
\ee
since there exists a constant $C_{\psi}$ such that $\Ik \leq C_{\psi}k^{-2\beta_0}\ll \frac{\sqrt{\kstar}}{n}$. Now, by Hoeffding decomposition, we have
\be 
\ & \ihatk-\Ik \\&=U_{n,*}-U_{n,0}-\Ef\left(U_{n,*}-U_{n,0}\right)\\
&=\frac{1}{n(n-1)}\sum_{i \neq j}\left[K_{*}(X_i,X_j)-K_0(X_i,X_j)-\Ef\left(K_{*}(X_i,X_j)-K_0(X_i,X_j)\right)\right]
\\&=\frac{2}{n}\sum_{i=1}^n\left[\mathbb{E}_{X_j}(K_{*}(X_i,X_j)-K_0(X_i,X_j))-\Ef\left(K_{*}(X_i,X_j)-K_0(X_i,X_j)\right)\right]
\\&+\frac{1}{n(n-1)}\sum_{i \neq j}\left[
\begin{array}{c}
K_{*}(X_i,X_j)-K_0(X_i,X_j)-\mathbb{E}_{X_j}(K_{*}(X_i,X_j)
-K_0(X_i,X_j))\\
-\mathbb{E}_{X_i}(K_{*}(X_i,X_j)-K_0(X_i,X_j))+\Ef\left(K_{*}(X_i,X_j)-K_0(X_i,X_j)\right)
\end{array}
\right]\\
&:=T_1+T_2
\ee
Therefore, 
\be
\ & \Pf\left[|\ihatk-\Ik|\geq C\sqrt{\log{n}}\frac{\sqrt{\kstar}}{n}\right]\\
&=\Pf\left[|T_1+T_2|\geq C\sqrt{\log{n}}\frac{\sqrt{\kstar}}{n}\right]\\
&\leq \Pf\left[|T_1|\geq C\delta_n\sqrt{\log{n}}\frac{\sqrt{\kstar}}{n}\right]+\Pf\left[|T_2|\geq C(1-\delta_n)\sqrt{\log{n}}\frac{\sqrt{\kstar}}{n}\right]
\ee
for some sequence $\delta_n \geq 0$ to be specified later. 
\subsubsection*{Control of $\Pf\left[|T_1|\geq C\delta_n\sqrt{\log{n}}\frac{\sqrt{\kstar}}{n}\right]$}
\be
\ &\Pf\left[|T_1|\geq C\delta_n\sqrt{\log{n}}\frac{\sqrt{\kstar}}{n}\right] \\&= \Pf\left[|\frac{2}{n}\sum_{i=1}^n\left[\mathbb{E}_{X_j}(K_{*}(X_i,X_j)-K_0(X_i,X_j))-\Ef\left(K_{*}(X_i,X_j)-K_0(X_i,X_j)\right)
	\right]|\geq C\delta_n\sqrt{\log{n}}\frac{\sqrt{\kstar}}{n}\right]
\ee
Let $R_i=\mathbb{E}_{X_j}(K_{*}(X_i,X_j)-K_0(X_i,X_j))$. Then the above display can be bounded from above by Markov's Inequality as follows,
\be
\ & \Pf\left[\frac{2}{n}|\sum_{i=1}^n\left(R_i-\Ef(R_i)\right)|\geq c\delta_n\sqrt{\log{n}}\frac{\sqrt{\kstar}}{n}\right]\\& \leq \frac{\Ef\left[\sum_{i=1}^n|R_i-\Ef(R_i)|\right]^{2r}}{\left(c\delta_n\sqrt{\log{n}}\frac{\sqrt{\kstar}}{n}n\right)^{2r}}\\
& \leq C_r n^{2r} \frac{1}{n^r}\sum_{i=1}^n\Ef\left[|R_i-\Ef(R_i)|^{2r}\right]\frac{1}{\left(c\delta_n\sqrt{\log{n}}\frac{\sqrt{\kstar}}{n}n\right)^{2r}}\\
& \leq C_r n^{2r} \frac{n}{n^r}\Ef\left[|R_i|^{2r}\right]\frac{1}{\left(c\delta_n\sqrt{\log{n}}\frac{\sqrt{\kstar}}{n}n\right)^{2r}}\\ \label{eqn:moment_first_order_degenrate_kernel}
\ee
where the last two lines is by applications of Marcinkiewicz – Zygmund Inequality followed by Jensen's Inequality. Next, we evaluate $\Ef(|R|^{2r})$.
\be
\Ef(|R|^{2r})&=\mathbb{E}_{f_{X_1}}\left[\left\{\mathbb{E}_{f_{X_2}}\left(K_{*}(X_1,X_2)-K_{0}(X_1,X_2)\right)\right\}^{2r}\right]\\
&\leq 2^{2r-1}\left[\mathbb{E}_{f_{X_1}}\left\{\mathbb{E}_{f_{X_2}}\left(K_{*}(X_1,X_2)\right)\right\}^{2r}+\mathbb{E}_{f_{X_1}}\left\{\mathbb{E}_{f_{X_2}}\left(K_{0}(X_1,X_2)\right)\right\}^{2r}\right]\\
\ee
Let us evaluate $\mathbb{E}_{f_{X_1}}\left\{\mathbb{E}_{f_{X_2}}\left(K_{k}(X_1,X_2)\right)\right\}^{2r}$ for a general $k$. We first consider,
\be
\mathbb{E}_{f_{X_2}}\left(K_{k}(X_1,X_2)\right)&=\sum_{l=1}^k{\psi_{k,l}(X_1)\mathbb{E}_{f_{X_2}}(\psi_{k,l}(X_2))}\\
&=\sum_{l=1}^k{\psi_{k,l}(X_1)\alpha_{k,l}}\leq (\sup_{l}\alpha_{k,l})\sum_{l=1}^k{|\psi_{k,l}(X_1)|}\\
& \leq C_{\psi_0}\|f\|_{\infty}^2
\ee
where the last line follows by Lemma \ref{lemma_wavelet_properties}. Therefore, for some constant $C_{\psi_0}>0$, 
$$\mathbb{E}_{f_{X_1}}\left\{\mathbb{E}_{f_{X_2}}\left(K_{k}(X_1,X_2)\right)\right\}^{2r}\leq C_{\psi_0}^{r}\|f\|_{\infty}^2.$$
Therefore from equation \ref{eqn:moment_first_order_degenrate_kernel}, we have
\be
\ & \Pf\left[\frac{2}{n}|\sum_{i=1}^n\left(R_i-\Ef(R_i)\right)|\geq c\delta_n\sqrt{\log{n}}\frac{\sqrt{\kstar}}{n}\right]\\&\leq C_r n^{2r} \frac{n}{n^r}\Ef\left[|R|^{2r}\right]\frac{1}{\left(c\delta_n\sqrt{\log{n}}\frac{\sqrt{\kstar}}{n}n\right)^{2r}}\\
& \leq C_{\psi_0}^{r}\|f\|_{\infty}^2 \frac{n^{r+1}}{\left(c\delta_n\sqrt{\log{n}}\sqrt{\kstar}\right)^{2r}} \leq  \frac{C_{\psi_0}^{r}\|f\|_{\infty}^2}{n}
\ee
by choosing $r$ large enough and $\delta_n>0$ at most a sub-algebraic sequence going to $0$. To see this note that $\frac{n^{r+1}}{\left(\delta_n\sqrt{\log{n}}\sqrt{\kstar}\right)^{2r}}=\frac{1}{\left(\delta_n\sqrt{\log{n}}\right)^{2r}}\frac{(\log{n})^{\frac{r}{1+4\beta_1}}}{n^{-(r+1)+\frac{2r}{1+4\beta_1}}}$. Choosing $r \geq \frac{2}{1-4\beta_1}$ yields the desire result.
\subsubsection*{\textbf{Control of $\mathbf{\Pf\left[|T_2|\geq C(1-\delta_n)\sqrt{\log{n}}\frac{\sqrt{\kstar}}{n}\right]}$}}

Lemma \ref{lemma_exponential_inequality_u_statistics} and Lemma \ref{lemma_evaluation_A1_to_A4} imply that for some deterministic constant $C_{\psi_0}$ and all $t>0$,
\be
\Pf\left[|\frac{1}{n(n-1)}\sum_{i \neq j}{H(X_i,X_j)}|> \frac{C_{\psi_0}}{n-1}\left(\sqrt{k_{*} t}+t+\sqrt{\frac{\kstar}{n}}t^{\frac{3}{2}}+\frac{\kstar}{n}t^2\right)\right]& \leq 6 e^{-t}
\ee
Using, $2t^{\frac{3}{2}}\leq t+t^2$ we have,
\be
\Pf\left[|T_2|> \frac{C_{\psi_0}}{n-1}\left(\sqrt{k_{*} t}+t+\sqrt{\frac{\kstar}{4n}}(t+t^2)+\frac{\kstar}{n}t^2\right)\right]& \leq 6 e^{-t}
\ee
Setting $t=\log{n}$ we have,
\be
\Pf\left[|T_2|> \frac{C_{\psi_0}}{n-1}\left(\sqrt{k_{*} \log{n}}+\log{n}+\sqrt{\frac{\kstar}{4n}}(\log{n}+(\log{n})^2)+\frac{\kstar}{n}(\log{n})^2\right)\right]& \leq \frac{6}{n}
\ee

Now, since $\frac{\sqrt{k_{*} \log{n}}}{n-1}\geq \max\left\{\frac{\log{n}}{n-1}, \frac{1}{n-1}\sqrt{\frac{\kstar}{4n}}(\log{n}+(\log{n})^2), \frac{\kstar}{n(n-1)}(\log{n})^2\right \}$ and $\frac{1}{n-1}>\frac{1}{2n}$ for sufficiently large $n$, we have,

\be
\Pf\left[|T_2|> 2C_{\psi_0}\frac{\sqrt{k_{*} \log{n}}}{n}\right]& \leq \frac{6}{n}
\ee
This is enough to prove the desired result and ends the proof of Lemma \ref{lemma_wrong_jhat}.
\end{proof}

\subsection*{Proof of Lemma \ref{lemma_error_lqnorm}}
\be
\Ef\left[\left(U_{n,1}-\phi(f)\right)^{2q}\right]& \leq 2^{2q-1}\Ef\left[\left(U_{n,1}-\Ef\left(U_{n,1}\right)\right)^{2q}\right]+2^{2q-1}\Ef\left[\left(\Ef\left(U_{n,1}-\phi(f)\right)\right)^{2q}\right]\\
&\leq 2^{2q-1}\Ef\left[\left(U_{n,1}-\Ef\left(U_{n,1}\right)\right)^{2q}\right]+C_{\psi_0}^{2q}\fmax^{2q}k_1^{-2q\beta_0}\\
&\leq 2^{2q-1}\Ef\left[\left(U_{n,1}-\Ef\left(U_{n,1}\right)\right)^{2q}\right]+C_{\psi_0}^{2q}\fmax^{2q}\left(\frac{k_1}{n^2}\right)^{q}
\ee
where the second last inequality follows from Lemma \ref{lemma_bias_truncated_ustatistics}. Therefore, it is enough to show that
\be
\sup_{f \in H(\beta_0)}\Ef\left[\left(U_{n,1}-\Ef\left(U_{n,1}\right)\right)^{2q}\right] \leq C_{\psi_0}^q \fmax^{Cq}\left(\frac{k_1}{n^2}\right)^q 
\label{eqn:qth_moment_ustat}
\ee
We will give two proofs of inequality \ref{eqn:qth_moment_ustat}. The first will be valid for any kernels based on compactly supported wavelet bases. The second will be valid only for Haar wavelets. 

\subsubsection*{\textbf{First Proof of Inequality \ref{eqn:qth_moment_ustat}}} \hspace*{\fill} \\
\be
\ & \Ef\left[\left(U_{n,1}-\Ef\left(U_{n,1}\right)^{2q}\right)\right] \\ & \leq C_q
\left\{\begin{array}{c}\Ef\left[\left(\frac{2}{n}\sumin \left\{\Ej(K_1(X_i,X_j))-\Ef(K_1(X_i,X_j))\right\}\right)^{2q}\right]\\+\Ef\left[\left(\frac{1}{n(n-1)}\sumij H_1(X_i,X_j)\right)^{2q}\right]\end{array}\right\}\\
&=C_q(I_1+I_2)
\ee
where $H_1(X_i,X_j)=K_1(X_i,X_j)-\Ei(K_1(X_i,X_j))-\Ej(K_1(X_i,X_j))+\Ef(K_1(X_i,X_j))$. We first bound $I_1$ by Marcinkiewicz – Zygmund Inequality. To this end, first write $R_i=\Ej(K_1(X_i,X_j))$. Therefore, by Marcinkiewicz – Zygmund Inequality and Jensen's Inequality, we have,
\be
I_1&=2^{2q}\Ef\left[\left(\frac{1}{n}\sumin \left\{R_i-\Ef(R_i)\right\}\right)^{2q}\right]\\
&\leq C_q n^{-2q/2}\frac{1}{n}\sumin \Ef\left[\left(R_i-\Ef(R_i)\right)^{2q}\right]\\
&\leq  2 C_q n^{-q} \Ef(|R_1|^{2q})
\ee 
Therefore, it is enough to control $\Ef(|R_i|^{2q})=\Ei\left[\mid Ej\left\{K_1(X_i,X_j)\right\} \mid^{2q}\right]$.
\be
\mid Ej\left\{K_1(X_i,X_j)\right\} \mid &=\mid \sumlk \psilk(X_i)\Ej(\psilk(X-j)) \mid\\
&=\mid \sumlk \alphalk \psilk(X_i) \mid \leq \sup (|\alphalk|) \sumlk |\psilk(X_i)|\\
&\leq C_{\psi_0}\fmax \frac{1}{\sqrt{k_1}}C_{\psi_0}\fmax \sqrt{k_1}=C_{\psi_0}^2\fmax^2
\ee
by Lemma \ref{lemma_wavelet_properties}. Therefore, 
$$\Ef(|R_i|^{2q})\leq C_{\psi_0}^2\fmax^2 $$
as well. This in turn implies 
$$I_1 \leq 2 C_q n^{-q} \Ef(|R_1|^{2q}) \leq 2 C_q C_{\psi_0}^2\fmax^2 n^{-q}\leq 2 C_q C_{\psi_0}^2\fmax^2 \left(\frac{k_1}{n^2}\right)^q$$
as required. 
Therefore, it is enough to show that 
\be
I_2\leq C_{\psi_0}^q \fmax^{Cq}\left(\frac{k_1}{n^2}\right)^q 
\ee 
i.e. we need to control $\Ef \left(\n2 \sumij H_1(X_i,X_j)\right)^{2q}$ where $H_1(X_i,X_j)=K_1(X_i,X_j)-\Ej(K_1(X_i,X_j))-\Ei(K_1(X_i,X_j))+\Ef(K_1(X_i,X_j))$. Following the arguments of Lemma \ref{lemma_exponential_inequality_u_statistics} and Lemma \ref{lemma_evaluation_A1_to_A4}, we can also show that for all $t>0$,
$$\Pf\left[|\n2 \sumij H_1(X_i,X_j)|>\frac{\const}{n-1}\left(\sqrt{k_1 t}+t+\sqrt{\frac{k_1}{n}}t^{\frac{3}{2}}+\frac{k_1}{n^2}t^2\right)\right]\leq 6e^{-t}.$$
Using $2t^{\frac{3}{2}}\leq t^2+t$, we have,
$$\Pf\left[|\n2 \sumij H_1(X_i,X_j)|>\frac{\const}{n-1}\left(\sqrt{k_1 t}+t\left(\sqrt{\frac{k_1}{4n}}+1\right)+\left(\sqrt{\frac{k_1}{4n}}+\frac{k_1}{n^2}\right)t^2\right)\right]\leq 6e^{-t}.$$
Calling, $\n2 \sumij H_1(X_i,X_j)=Z$, we have an exponential inequality of the form,
$$\Pf\left[|Z|>a_1\sqrt{t}+a_2 t +a_3 t^2\right]\leq 6 e^{-t},$$
with $a_1=\frac{\const}{n-1}\sqrt{k_1},a_2=\frac{\const}{n-1}\left(\sqrt{\frac{k_1}{4n}}+1\right),a_3=\frac{\const}{n-1}\left(\sqrt{\frac{k_1}{4n}}+\frac{k_1}{n^2}\right)$. From this we need to estimate,
$$\Ef(|Z|^{2q})=2q\int_0^{\infty}x^{2q-1}\Pf(|Z|\geq x)dx.$$
We do this as follows. Suppose, $t(x)$ solve for $a_1\sqrt{t(x)}+a_2 t(x) +a_3 t^2(x)=x$. Now, trivially $a_1\sqrt{t}+a_2 t +a_3 t^2$ is an increasing function of $t \in \mathbb{R}^{+}$. Therefore, if we can find a function $h(x)\geq 0$ such that $a_1\sqrt{h(x)}+a_2 h(x) +a_3 h^2(x)\leq x$ for all $x \geq 0$, then $h(x)\leq t(x)$ for all $x \geq 0$.It is not too difficult to see that one such $h(x)$ is given by $h(x)=b_1 x^2 \wedge b_2 x \wedge b_3 \sqrt{x}$ where for $b_1=\frac{c_1}{a_1^2}, b_2=\frac{c_2}{a_2},b_3=\frac{c_3}{\sqrt{a_3}} $ for fixed constants $c_1,c_2,c_3>0$. Therefore,
\be
\Ef(|Z|^{2q})&=2q\int_0^{\infty}x^{2q-1}\Pf(|Z|\geq x)dx\\
&\leq 2q \int_0^{\infty}x^{2q-1}\Pf(|Z|\geq a_1\sqrt{h(x)}+a_2 h(x) +a_3 h^2(x))dx\\
&\leq 12q \int_0^{\infty}x^{2q-1} e^{-h(x)}dx\\
& = 12q \int_0^{\infty}x^{2q-1} e^{-\left\{b_1 x^2 \wedge b_2 x \wedge b_3 \sqrt{x}\right\}}dx\\
&\leq 12q \left[\int_0^{\infty}x^{2q-1} e^{-b_1 x^2}dx+\int_0^{\infty}x^{2q-1} e^{- b_2 x}dx+\int_0^{\infty}x^{2q-1} e^{-b_3 \sqrt{x}}dx\right]\\
&=12q\left(\frac{\Gamma(q)}{2b_1^{q}}+\frac{\Gamma(2q)}{b_2^{2q}}+\frac{2\Gamma(4q)}{b_3^{4q}}\right)\leq \const^q \left(\frac{k_1}{n^2}\right)^{q}
\ee
by our choices of $b_1,b_2,b_3$. This completes the first proof.

\subsection*{Proof of Lemma \ref{lemma_l_less_than_lc}}
Pick any $l \leq l_c(C^*)$. Then 
\be
\ & \Pf(\jhat=l)\leq \Pf(\hat{I}^2(k_l^*,k_j^*)\leq C_{\mathrm{opt}}^2 \log{n} R(k_j^*))\\
&=\Pf\left({\hat{I}(k_l^*,k_j^*)-I(k_l^*,k_j^*)}\leq C_{\mathrm{opt}} \sqrt{\log{n}}\sqrt{R(k_j^*)}-{I(k_l^*,k_j^*)}\right)\\&-\Pf\left({\hat{I}(k_l^*,k_j^*)-I(k_l^*,k_j^*)}<- C_{\mathrm{opt}} \sqrt{\log{n}}\sqrt{R(k_j^*)}-{I(k_l^*,k_j^*)}\right)\\ \label{eqn:l_less_than_lc}
\ee
Recall that by $l\leq l_c(C^*)$, we have that $I^2(k_l^*,k_j^*)>C^* \log{n} R(k_j^*)$. We consider two cases.
\subsubsection*{\textbf{Case 1: $\mathbf{I(k_l^*,k_j^*)>\sqrt{C^* \log{n} R(k_j^*)}}$}}\hspace*{\fill} \\
In this case we have that the right hand side of \eqref{eqn:l_less_than_lc} is bounded by 
\be
\ & \Pf\left({\hat{I}(k_l^*,k_j^*)-I(k_l^*,k_j^*)}\leq C_{\mathrm{opt}} \sqrt{\log{n}}\sqrt{R(k_j^*)}-\sqrt{C^* \log{n} R(k_j^*)}\right)\\
&+\Pf\left({\hat{I}(k_l^*,k_j^*)-I(k_l^*,k_j^*)}<- C_{\mathrm{opt}} \sqrt{\log{n}}\sqrt{R(k_j^*)}-\sqrt{C^* \log{n} R(k_j^*)}\right)\\
& \leq \frac{\const}{n},
\ee
where the last line follows from the proof of Lemma \ref{lemma_wrong_jhat} along the lines of Lemma \ref{lemma_exponential_inequality_u_statistics} and \ref{lemma_evaluation_A1_to_A4}, provided $C_{\mathrm{opt}}$ is according to the constant specified in Theorem \ref{theorem_two_point_quadratic} and $C^*>4C_{\mathrm{opt}}^2$.
\subsubsection*{\textbf{Case 2: $\mathbf{I(k_l^*,k_j^*)<-\sqrt{C^* \log{n} R(k_j^*)}}$}}\hspace*{\fill} \\
In this case we have that the right hand side of \eqref{eqn:l_less_than_lc} equals,
\be
\ & 1-\Pf\left({\hat{I}(k_l^*,k_j^*)-I(k_l^*,k_j^*)}> C_{\mathrm{opt}} \sqrt{\log{n}}\sqrt{R(k_j^*)}-\sqrt{C^* \log{n} R(k_j^*)}\right)\\
&-1+\Pf\left({\hat{I}(k_l^*,k_j^*)-I(k_l^*,k_j^*)}\geq- C_{\mathrm{opt}} \sqrt{\log{n}}\sqrt{R(k_j^*)}-\sqrt{C^* \log{n} R(k_j^*)}\right)\\
& \leq \frac{\const}{n},
\ee
where the last line once again follows from the proof of Lemma \ref{lemma_wrong_jhat} along the lines of Lemma \ref{lemma_exponential_inequality_u_statistics} and \ref{lemma_evaluation_A1_to_A4}, provided $C_{\mathrm{opt}}$ is according to the constant specified in Theorem \ref{theorem_two_point_quadratic} and $C^*>4C_{\mathrm{opt}}^2$.

\subsection*{Proof of Lemma \ref{lemma_l_bigger_than_jplusone}}
\begin{proof}
\be
\Pf(\jhat \geq j+2)&\leq \Pf\left(\exists l>j: \hat{I}^2(k_{j+1}^*,k_l^*)>C_{\mathrm{opt}}^2 \log{n} R(k_l^*)\right)\\
& \leq \sum\limits_{l=j+2}^{N-1}\Pf\left(\hat{I}^2(k_{j+1}^*,k_l^*)>C_{\mathrm{opt}}^2 \log{n} R(k_l^*)\right)
\ee
Now, for any $l\geq j+2$,
\begin{eqnarray*}
\ & \Pf\left(\hat{I}^2(k_{j+1}^*,k_l^*)>C_{\mathrm{opt}}^2 \log{n} R(k_l^*)\right)\\&= \Pf\left({\hat{I}(k_{j+1}^*,k_l^*)-I(k_{j+1}^*,k_l^*)}> C_{\mathrm{opt}} \sqrt{\log{n}}\sqrt{R(k_l^*)}-{I(k_{j+1}^*,k_l^*)}\right)\\
&+\Pf\left({\hat{I}(k_{j+1}^*,k_l^*)-I(k_{j+1}^*,k_l^*)}<- C_{\mathrm{opt}} \sqrt{\log{n}}\sqrt{R(k_l^*)}-{I(k_{j+1}^*,k_l^*)}\right).
\end{eqnarray*}
Thereafter, note that $\frac{I^2(k_{j+1}^*,k_l^*)}{R(k_l^*)}\leq \const \frac{(k_{j+1}^*)^{-4\beta_f}}{\frac{k_l^*}{n^2}}$, which has a power of $n$ equal to $\frac{8\beta_l}{1+4\beta_l}-\frac{8\beta_f}{1+4\beta_{j+1}}\leq 0$ since $\beta_l<\beta_{j+1}\leq \beta_f$. Hence arguing as  in proof of Lemma \ref{lemma_wrong_jhat} along the lines of Lemma \ref{lemma_exponential_inequality_u_statistics} and \ref{lemma_evaluation_A1_to_A4}, we have the desired result provided $C_{\mathrm{opt}}$ is according to the constant specified in Theorem \ref{theorem_two_point_quadratic}.
\end{proof}

\subsection*{Proof of Lemma \ref{lemma_control_of_alphanm}}
\begin{proof}
We begin by noting that it is enough to control 
\be
\int_{\chinm} \int_{\chinm}\int_{\chinm} \int K_{k_1}(x,x_1)K_{k_2}(x,x_2)K_{k_3}(x,x_3)f(x_1)f(x_2)f(x_3)dx dx_1 dx_2 dx_3
\ee
for arbitrary $M_n\leq k_1\leq k_2 \leq k_3$. We assume that $k_1$ divides $k_2$, $k_2$ divides $k_3$ and $M_n$ divides $k_1$. The proof for a general tuple follows by similar arguments with suitable obvious modifications. The above integral equals,
\be
\int_{\chinm} \int_{\chinm}\int_{\chinm} \int \sum_{l_1}\sum_{l_2}\sum_{l_3}\left\{\begin{array}{c} \psi_{l_1}^{k_1}(x)\psi_{l_1}^{k_1}(x_1)\psi_{l_2}^{k_2}(x)\\
	\psi_{l_2}^{k_2}(x_2)\psi_{l_3}^{k_3}(x)\psi_{l_3}^{k_3}(x_3)\end{array}\right\}\prod_{i=1}^3f(x_i)dx_i
\ee
where $\psi_{l}^{k}(x)$ is $k$ dilated an $l$ shifted wavelet bases. Taking the integral inside the summation, any of the summand equals at most $\constf \frac{1}{k_3}k_1 k_2 k_3\frac{1}{k_1 k_2 k_3}=\frac{\constf}{k_3}$. The reason being, each of the summand corresponds to an integral of $x_1, x_2,x_3$ over intervals of length $\frac{1}{k_1}$, $\frac{1}{k_2}$, and $\frac{1}{k_3}$ respectively.  Moreover, since $k_3$ is the finest refinement, the integral over x is simply over an interval of length $\frac{1}{k_3}$. The bound on the summand then follows. Therefore, we now need to count the number of summands that contribute to the sum above. This can be argued as follows. For every given subinterval $\chinm$ there are less than $\const \frac{k_1}{M_n}$ subintervals of length $\frac{\const}{k_1}$ with support intersecting that of some $\psi_{l_1}^{k_1}$. For each given subinterval of length $\frac{\const}{k_1}$ with support intersecting that of some $\psi_{l_1}^{k_1}$, there are less than $\const \frac{k_2}{k_1}$ subintervals of length $\frac{\const}{k_2}$ with support intersecting that of some $\psi_{l_2}^{k_2}$. Finally, for each given subinterval of length $\frac{\const}{k_2}$ with support intersecting that of some $\psi_{l_2}^{k_2}$, there are less than $\const \frac{k_3}{k_2}$ subintervals of length $\frac{\const}{k_3}$ with support intersecting that of some $\psi_{l_3}^{k_3}$. Therefore, total number of terms contributing to the sum above is at most $\const \frac{k_1}{M_n}\frac{k_2}{k_1}\frac{k_3}{k_2}=\const\frac{k_3}{M_n}$. This implies that in absolute value, each $\alpha_{n,m}$ is at most $\constf \frac{1}{k_3}\frac{k_3}{M_n}=\constf \frac{1}{M_n}$, as promised.
\end{proof}

\subsection*{Proof of Lemma \ref{lemma_moment_third_order_U_stat}}

\begin{proof}
	We evoke Proposition 2.4 of \cite{gine2000exponential} (page 9) which implies	that for a universal constant depending on $q$
	\be
		\ & E\left\vert \frac{1}{m\left( m-1\right) \left( m-2\right) }%
		\sum\limits_{i_{1}\neq i_{2}\neq i_{3}}K\left(
		X_{i_{1}},X_{i_{2}},X_{i_{3}}\right) \right\vert ^{q}   \\
		&\lesssim \left( m\left( m-1\right) \left( m-2\right) \right)
		^{-q}C_{q}\left\{ 
		\begin{array}{c}
			m^{3q/2}E\left[ \left\vert E\left( K^{2}\left(
			X_{i_{1}},X_{i_{2}},X_{i_{3}}\right) |X_{i_{3}}\right) \right\vert ^{q/2}%
			\right] \\ 
			+m^{q}E\left[ \max_{i_{3}}\left\vert E\left( K^{2}\left(
			X_{i_{1}},X_{i_{2}},X_{i_{3}}\right) |X_{i_{3}}\right) \right\vert ^{q/2}%
			\right] \\ 
			+m^{q/2}E\left[ \max_{i_{3},i_{2}}\left\vert E\left( K^{2}\left(
			X_{i_{1}},X_{i_{2}},X_{i_{3}}\right) |X_{i_{3}},X_{i_{2}}\right) \right\vert
			^{q/2}\right] \\ 
			+E\left[ \max_{i_{3},i_{2},i_{1}}\left\vert \left( K\left(
			X_{i_{1}},X_{i_{2}},X_{i_{3}}\right) \right) \right\vert ^{q}\right]%
		\end{array}%
		\right\} \\ \label{eqn:third_order_bound}
	\ee
	
	Now, we have that for $Z_{1},...,Z_{m}>0$ identically distributed, possibly
	dependent with $E\left( Z_{\left( m\right) }\right) \leq mE\left(
	Z_{1}\right) =E\left( \sum\limits_{i}Z_{i}\right) $ where $Z_{\left(
		m\right) }=\max_{i\in \left\{ 1,...,m\right\} }Z_{i}.$ \ So that \ref{eqn:third_order_bound} is bounded above by 
	\begin{eqnarray*}
		&&C_{q}\left( m\left( m-1\right) \left( m-2\right) \right) ^{-q}C_{q}\left\{ 
		\begin{array}{c}
			m^{3q/2}E\left[ \left\vert K\left( X_{i_{1}},X_{i_{2}},X_{i_{3}}\right)
			\right\vert ^{q}\right] \\ 
			+m^{q+1}E\left[ \left\vert K\left( X_{i_{1}},X_{i_{2}},X_{i_{3}}\right)
			\right\vert ^{q}\right] \\ 
			+m^{q/2+2}E\left[ \left\vert K\left( X_{i_{1}},X_{i_{2}},X_{i_{3}}\right)
			\right\vert ^{q}\right] \\ 
			+m^{3}E\left[ \left\vert K\left( X_{i_{1}},X_{i_{2}},X_{i_{3}}\right)
			\right\vert ^{q}\right]%
		\end{array}%
		\right\} \\
		&\lesssim &C_{q}m^{-q}E\left[ \left\vert K\left(
		X_{i_{1}},X_{i_{2}},X_{i_{3}}\right) \right\vert ^{q}\right]
	\end{eqnarray*}
\end{proof}

\section{Technical Lemmas}
The following lemma regarding the bias of $U_n^{k}$ will be used regularly.
\begin{lemma}\label{lemma_bias_truncated_ustatistics}
	Suppose $f \in H(\beta)$. Then for compactly supported wavelet bases 
	$$\sup_{f}\mid \Ef(U_n^k)-\phi(f)\mid \leq C_{\psi_0}\fmax k^{-2\beta_f}.$$
\end{lemma}
\begin{proof}
The proof  involves simple algebra and properties of compactly supported wavelet bases \citep{hardle1998wavelets} and hence is omitted.

We will be using the following lemma about properties of compactly supported wavelets.
\begin{lemma}\label{lemma_wavelet_properties}
	For kernels based on compactly supported wavelets, as defined by Equation \eqref{eqn:wavelet_expansion}, the following hold.
	\begin{enumerate}
		\item $\|f_k\|_{\infty}\leq C_{\psi_0}\|f\|_{\infty}.$
		\item $\sup_{l}|\alpha_{k,l}|\leq \frac{C_{\psi_0}\|f\|_{\infty}}{k}.$
		\item $\sup_{x}\sum_{l=1}^k|\psi_{k,l}(x)|\leq  C_{\psi_0}. k$
	\end{enumerate}
\end{lemma}
\begin{proof}
	Once again, the proofs follow from simple algebra and properties of compactly supported wavelet bases \citep{hardle1998wavelets}.
\end{proof}
\end{proof}
We will also need the following asymptotic normality of the quadratic estimators \citep{rltv2015}. 
\begin{lemma}\label{lemma_asymp_norm}
	For $f \in H(\beta,c)$ we have $\frac{U_n^{k(\beta)}-\Ef\left(U_n^{k(\beta)}\right)}{\sqrt{Var_f\left(U_n^{k(\beta)}\right)}} \Rightarrow N(0,1)$ where $Var_f\left(U_n^{k(\beta)}\right) \sim \frac{k(\beta)}{n^2}$.
\end{lemma}

%\begin{lemma}\label{lemma_sup_wavelet_coeffs}
%	For kernels based on compactly supported wavelet bases,
%	$$\sup_{k,l}\alpha_{k,l}\leq \frac{C_{\psi_0}\|f\|_{\infty}}{k}$$
%\end{lemma}

%\begin{lemma}\label{lemma_sup_sum_wavelets}
%	For kernels based on compactly supported wavelet bases,
%	$$\sup_{x}\sum_{l=1}^k|\psi_{k,l}(x)|\leq  C_{\psi_0} k$$
%\end{lemma}
We will need the following lemma about tails of second order U-statistics. %Before we state the lemma, let us introduce the following notation for the kernel of the degenerate U-statistics $T_2$. 
%\newpage

\begin{lemma}\label{lemma_exponential_inequality_u_statistics}
	Let,
	\be
	H(X,Y)&=K_{*}(X,Y)-K_0(X,Y)-\mathbb{E}_{Y}(K_{*}(X,Y)
	-K_0(X,Y))\\
	&-\mathbb{E}_{X}(K_{*}(X,Y)-K_0(X,Y))+\Ef\left(K_{*}(X,Y)-K_0(X,Y)\right)
	\ee.
	Then the following exponential inequality holds for every $t>0$ and a deterministic constant $C>0$.
	$$\Pf\left[|\sum_{i \neq j}{H(X_i,X_j)}|\geq C(A_1\sqrt{t}+A_2 t^2+A_3 t^{\frac{3}{2}}+A_4 t^2)\right]\leq 5.6 e^{-t},$$
	where
	$$A_1^2=n(n-1)\Ef[H^2(X_1,X_2)],$$
	\be
	A_2&=\sup\left\{\begin{array}{c}|\Ef\left[\sum_{i=1}^n\sum_{j=1}^{i-1}{H(X_1,X_2)a_i(X_1)b_j(X_2)}\right]|:\\\Ef\left(\sum_{i=1}^n{a_i^2(X_1)}\right)\leq 1, \Ef\left(\sum_{j=1}^n{b_j^2(X_2)}\right)\leq 1\end{array} \right\},
	\ee
	$$A_3^2=n \sup_x\{\mathbb{E}_{f_{X_2}}(H^2(x,X_2))\},$$
	$$A_4=\sup_{x,y}|H(x,y)|.$$
\end{lemma}
\begin{proof}
	This follows  directly from Theorem 3.4 of \cite{houdre2003exponential}.
\end{proof}
We now estimate the terms $A_1, A_2, A_3, A_4$ in the following lemma.

\begin{lemma}\label{lemma_evaluation_A1_to_A4}
	For kernels based on compactly supported wavelet bases, there exists deterministic constant $C_{\psi_0}$ (depending on $\psi_0$ and $\|f\|_{\infty}$) such that,
	$$A_1^2 \leq C_{\psi_0}^2 n(n-1)(k_{*}+k_0), A_2 \leq C_{\psi_0} n, A_3^2 \leq C_{\psi_0}n(k_{*}+k_0), A_4 \leq  C_{\psi_0}(k_{*}+k_0).$$
\end{lemma}
\begin{proof}
	The proof follows along the same line of arguments as those laid down in Proposition 2 of \cite{bull2013adaptive}.
\end{proof}

\bibliographystyle{imsart-nameyear}
\bibliography{biblio_adaptation}

\end{document}